\documentclass[11pt]{amsart}
\usepackage{amssymb,amsmath}
\usepackage{graphicx}
\usepackage{bm,bbm}
\usepackage{color}
\usepackage{tikz}
\tikzset{
	pics/carc/.style args={#1:#2:#3}{
		code={
			\draw[pic actions] (#1:#3) arc(#1:#2:#3);
		}
	}
}
\usepackage{ifthen}
\usepackage{pgfplots}
\usepackage{array,blkarray}
\usepackage[hidelinks]{hyperref}
\usepackage{mathtools}  
\newtheorem{theorem}{Theorem}
\newtheorem{proposition}[theorem]{Proposition}
\newtheorem{corollary}[theorem]{Corollary}
\theoremstyle{definition}

\newtheorem{example}[theorem]{Example}
\theoremstyle{lemma}
\newtheorem{lemma}[theorem]{Lemma}

\theoremstyle{remark}
\newtheorem{remark}[theorem]{Remark}
\newtheorem{assumption}[theorem]{Assumption}
\numberwithin{theorem}{section}
\numberwithin{equation}{section}
\numberwithin{table}{section}
\numberwithin{figure}{section}
\newcommand{\M}{\ensuremath{\mathcal{M}}}
\newcommand{\Q}{\ensuremath{\mathcal{Q}}}
\newcommand{\V}{\ensuremath{\mathcal{V}}}
\newcommand{\Vkl}{\ensuremath{\mathbb{V}}}
\newcommand{\Mh}{\ensuremath{M_H}}
\newcommand{\Qh}{\ensuremath{Q_H}}
\newcommand{\tQh}{\ensuremath{\widetilde Q_H}}
\newcommand{\Vh}{\ensuremath{V_H}}
\newcommand{\Vhkl}{\ensuremath{{\mathbb V}_h}}
\newcommand{\dimu}{{\ensuremath{n_u}}}
\newcommand{\dimp}{{\ensuremath{n_p}}}
\newcommand{\diml}{{\ensuremath{n_\lambda}}}
\newcommand{\calTO}{\ensuremath{\mathcal{T}_\Omega}}
\newcommand{\HO}{\ensuremath{H_\Omega}}
\newcommand{\calTG}{\ensuremath{\mathcal{T}_\Gamma}}

\newcommand{\uh}{\ensuremath{u_H}}
\newcommand{\duh}{\ensuremath{\dot u_H}}
\newcommand{\ph}{\ensuremath{p_H}}
\newcommand{\dph}{\ensuremath{\dot p_H}}
\newcommand{\lh}{\ensuremath{\lambda_H}}
\newcommand{\fraka}{\ensuremath{\mathfrak{a}}}
\newcommand{\frakaTilde}{\ensuremath{\widetilde{\mathfrak{a}}}}
\newcommand{\frakb}{\ensuremath{\mathfrak{b}}}
\newcommand{\frakk}{\ensuremath{\mathfrak{K}}}
\newcommand{\calRu}{\ensuremath{u^\calR_H}} 
\newcommand{\calRp}{\ensuremath{\tilde p^\calR_H}} 
\newcommand{\calRl}{\ensuremath{\lambda^\calR_H}}
\newcommand{\calRdu}{\ensuremath{\dot u^\calR_H}} 
\newcommand{\calRdp}{\ensuremath{\dot p^\calR_H}} 

\def\R{\mathbb{R}}

\definecolor{myBlue2}{RGB}{113,104,238} 
\definecolor{myBlue3}{RGB}{30,144,255} 
\definecolor{myGreen2}{RGB}{69,169,0} 
\definecolor{myGreen3}{RGB}{154,205,50} 
\definecolor{myRed2}{RGB}{165,42,42} 
\definecolor{color0}{rgb}{0.12156862745098,0.466666666666667,0.705882352941177}
\definecolor{color1}{rgb}{1,0.498039215686275,0.0549019607843137}
\definecolor{color3}{rgb}{0.83921568627451,0.152941176470588,0.156862745098039}
\definecolor{color2}{rgb}{0.172549019607843,0.627450980392157,0.172549019607843}
\DeclareMathOperator{\id}{id}

\DeclareMathOperator{\wcba}{wcba}

\newcommand{\calA}{\ensuremath{\mathcal{A}}}
\newcommand{\calB}{\ensuremath{\mathcal{B}}}

\newcommand{\calE}{\ensuremath{\mathcal{E}}}

\newcommand{\calG}{\ensuremath{\mathcal{G}}}

\newcommand{\calK}{\ensuremath{\mathcal{K}}}

\newcommand{\calM}{\ensuremath{\mathcal{M}}}
\newcommand{\calN}{\ensuremath{\mathcal{N}}}
\newcommand{\calO}{\mathcal{O}}
\newcommand{\calP}{\ensuremath{\mathcal{P}}}
\newcommand{\calQ}{\ensuremath{\mathcal{Q}}}
\newcommand{\calR}{\ensuremath{\mathcal{R}}}

\newcommand{\calT}{\ensuremath{\mathcal{T}}}

\newcommand{\calV}{\ensuremath{\mathcal{V}}}
\newcommand{\calW}{\ensuremath{\mathcal{W}}}

\def\dt{\,\text{d}t}
\def\dtau{\,\text{d}\tau}
\def\dx{\,\text{d}x}

\newcommand{\ddt}{\ensuremath{\frac{\text{d}}{\text{d}t}} }

\newcommand{\wconv}{\rightharpoonup}
\def\eps{\varepsilon}
\begin{document}
\title[A Multiscale Method for Heterogeneous Bulk-Surface Coupling]{A Multiscale Method for Heterogeneous\\ Bulk-Surface Coupling$^\star$}
\author[]{R.~Altmann$^\dagger$, B.~Verf\"urth$^{\ddagger}$}
\address{${}^{\dagger}$ Institut f\"ur Mathematik, Universit\"at Augsburg, Universit\"atsstr.~14, 86159 Augsburg, Germany.}\email{robert.altmann@math.uni-augsburg.de}
\address{${}^{\ddagger}$ Institut f\"ur Angewandte und Numerische Mathematik, Karlsruher Institut f\"ur Technologie, Englerstr.~2, 76131 Karlsruhe, Germany.}
\email{barbara.verfuerth@kit.edu}
\thanks{${}^\star$ RA's work is funded by the Deutsche Forschungsgemeinschaft (DFG, German Research Foundation) -- Project-ID 446856041. BV was affiliated with University of Augsburg when major parts of the work were carried out. BV's work at KIT is funded by the Deutsche Forschungsgemeinschaft (DFG, German Research Foundation) -- Project-ID 258734477 -- SFB 1173.}
%
\date{}
%
%
\begin{abstract}
In this paper, we construct and analyze a multiscale (finite element) method for parabolic problems with heterogeneous dynamic boundary conditions. As origin, we consider a reformulation of the system in order to decouple the discretization of bulk and surface dynamics. This allows us to combine multiscale methods on the boundary with standard Lagrangian schemes in the interior. 
We prove convergence and quantify explicit rates for low-regularity solutions, independent of the oscillatory behavior of the heterogeneities. As a result, coarse discretization parameters, which do not resolve the fine scales, can be considered. 
The theoretical findings are justified by a number of numerical experiments including dynamic boundary conditions with random diffusion coefficients. 
\end{abstract}
%
\maketitle
%
%
%
\section{Introduction}
This paper is concerned with coupled bulk-surface partial differential equations (PDE) with a heterogeneous medium considered on the surface, modelled through dynamic boundary conditions. Problems with dynamic boundary conditions arise, e.g., as limit of a coupled bulk-bulk problem with a thin outer domain~\cite{Lie13}. If this outer domain is additionally heterogeneous, then this transfers to the limiting boundary conditions. 
Further applications include fluid-structure and acoustic-elastic interaction if one component can be modelled in form of a boundary layer~\cite{Hip17}. In this way, generalized boundary conditions can simplify and reduce models used for example in hemodynamics, modeling the blood flow in arteries~\cite{FigVJHT06}. Dynamic boundary conditions also enable a proper way to model a heat source or a heat transfer on the boundary~\cite{Esc93,Gol06}. Generally speaking, dynamic boundary conditions are of high significance if one needs to reflect the effective properties of the surface. 

Although the inclusion of dynamic boundary conditions is well-understood from a theoretical point of view, see e.g.~\cite{FavGGR02,CocFGGR08,VazV08}, the corresponding numerical analysis drags behind. There are only a handful of papers dealing with the numerical approximation of such (or related) problems. 
For stationary elliptic problems an isoparametric finite element method was introduced in~\cite{EllR13}. 
Numerical approximation schemes for parabolic problems with dynamic boundary conditions are
presented in~\cite{VraS13a,KovL17}. In both cases, a standard Galerkin ansatz for the spatial discretization is considered, i.e., the mesh on the boundary is automatically
specified through the restriction of the mesh of the bulk to the boundary. 
Such approaches, however, suffer if the solution is oscillatory on the boundary or does not contain a sufficiently regular trace, e.g., due to heterogeneities on the boundary.

In this paper, we propose an alternative approach based on a formulation as a coupled system, cf.~\cite[Ch.~5.3]{Las02}. This means that bulk and surface dynamics are considered as two systems, which are coupled through the boundary. Similar approaches were taken in~\cite{EngF05} for theoretical purposes in the semigroup setting or in the framework of dual continuum models~\cite{Lic00} used in the field of fractured porous media.
We consider the weak formulation of the problem and the interpretation as partial differential-algebraic equation (PDAE). This system class provides a powerful framework (especially in terms of modeling) for general coupled systems, see~\cite{KunM06,LamMT13,Alt15}. 
The PDAE formulation comes along with a saddle point structure and thus, needs a special treatment in order to prevent numerical instabilities. More precisely, we need to design inf-sup stable schemes, leading to a novel class of mixed finite element methods. Here we allow independent discretizations in the bulk and on the surface. In this sense, these methods reveal a flexibility known from non-conforming schemes although formulated in a conforming framework. 

The possibility to combine coarse grids in the interior with fine grids or adapted schemes on the boundary is of great value if bulk and surface dynamics have different characteristic length scales. Assuming heterogeneous diffusion coefficients on the boundary without scale separation, we propose to apply the Localized Orthogonal Decomposition (LOD) on the boundary. This method was originally introduced in~\cite{MalP14} for elliptic problems and further developed in the last years covering a large range of applications.
In view of this paper, we particularly mention the application to standard parabolic problems~\cite{MalP18}, thermoelasticity~\cite{MalP17}, and poroelasticity~\cite{AltCMPP20}.
Furthermore, \cite{HelMW19} recently discussed the application to fracture problems, where also a bulk problem is coupled to an interface problem. Therein, however, the multiscale features are relevant in the bulk as well and the problems in the volume and on the interface remain completely coupled. 
The LOD can also be interpreted in the context of subspace decomposition methods~\cite{KorY16,KorPY18}.
Furthermore, it is closely connected to the area of homogenization~\cite{GalP17,PetVV19}, which allows an alternative interpretation of our method, in particular for one-dimensional boundaries. 

The combination of a multiscale method on the boundary and standard Lagrangian schemes in the bulk allows for a computationally efficient and accurate representation of the coarse dynamics for heterogeneous bulk-surface problems.
We prove convergence of the corresponding semi-discrete scheme with explicit rates even for low-regularity solutions as they may appear for general heterogeneous and discontinuous surface diffusion coefficients.
Furthermore, the multiscale method applied on the boundary leads to accurate approximations and convergence rates even in the pre-asymptotic regime, i.e., when the oscillations and jumps of the coefficients are completely unresolved by the mesh.

The paper is structured as follows. 
In Section~\ref{sect:model} we derive the system equations with heterogeneous dynamic boundary conditions as the limit of a coupled bulk-bulk problem. Further, we discuss two possible weak formulations. 
The novel class of discretization schemes is introduced in Section~\ref{sect:disc}. As it is based on a coupled formulation, we consider a special class of mixed finite elements in combination with the LOD. 
A specific multiscale method is then presented and analyzed in Section~\ref{sect:errAnalysis}. Numerical evidence of the theoretical results, clearly showing the computational gains of the approach, are subject of Section~\ref{sect:numerics}.
%
%
\section{Derivation of Dynamic Boundary Conditions}\label{sect:model}
In this section, we derive the system equations for the considered heterogeneous bulk-surface coupling. 
For this, we first motivate the dynamic boundary conditions as the limit of a bulk-bulk coupling and discuss corresponding weak formulations afterwards. 
In order to deal with the heterogeneity on the boundary later on, we consider a decoupled approach, which is beneficial for the numerical consideration. 
%
%
\subsection{Dynamic boundary conditions as a limit}\label{sect:model:limit}
We consider the linear heat equation with constant thermal diffusivity~$\kappa>0$ in a bounded domain~$\Omega\subseteq\R^d$, $d\ge2$, coupled with a second parabolic problem in the surrounding domain~$\Omega_\delta$ of thickness~$\delta>0$. 
The joint boundary is denoted by~$\Gamma\coloneqq\overline{\Omega} \cap \overline{\Omega}_\delta$. 
Moreover we assume the outer material to be heterogeneous in tangential direction and constant in normal direction, which is encoded in the diffusion coefficient~$a_\eps\in L^\infty(\Omega_\delta)$, cf.~Figure~\ref{fig:domains}. 
More precisely, we assume~$a_\eps$ to be of the form 
\[
  a_\eps (x)
  = a_\eps(P_\delta(x)),
\]
where $P_\delta(x)\in \Gamma$ denotes the normal projection of~$x\in \Omega_{\delta}$ onto the boundary~$\Gamma$. 
Here, $\eps\ll 1$ is a small parameter, which corresponds to the oscillatory behavior of the diffusion. In the special case of a periodic coefficient, $\eps$ equals the period length.
Furthermore, we assume $a_\eps$ to be uniformly bounded from below by a positive constant~$\alpha>0$.
\begin{figure}
\pgfmathsetseed{\number\pdfrandomseed}
\begin{tikzpicture}[scale=1.1]
	\foreach \phi in {0, 10, ..., 350} {
		\pgfmathparse{rnd}
		\pgfmathsetmacro{\rndRed}{\pgfmathresult}
		\pgfmathsetmacro{\rndGreen}{\pgfmathresult}
		\definecolor{rndCol}{rgb}{\rndRed,\rndGreen,0.6}
		\draw[rndCol, line width=10pt] pic{carc=\phi:\phi+10:1.5};		
		\draw[rndCol, line width=2pt, xshift=7cm] pic{carc=\phi:\phi+10:1.5};		
	}
	\draw[<->] pic{carc=130:140:1.85};		
	\node at (-1.3, 1.3) {\small $\eps$};	
	\node at (0, 0) {$\Omega$};	
	\node at (1.5, 1.1) {$\Omega_{\delta}$};	
	\node at (7, 0) {$\Omega$};	
	\node at (8.3, 1.1) {$\Gamma$};	
	\draw[thick, ->] (2.1, 0) -- (5.0, 0);	
	\node at (3.5, 0.25) {$\delta \to 0$};
\end{tikzpicture}
\caption{Illustration of the domains $\Omega$, $\Omega_\delta$, the boundary~$\Gamma$, and the limiting process $\delta\to 0$.} 
\label{fig:domains}
\end{figure}
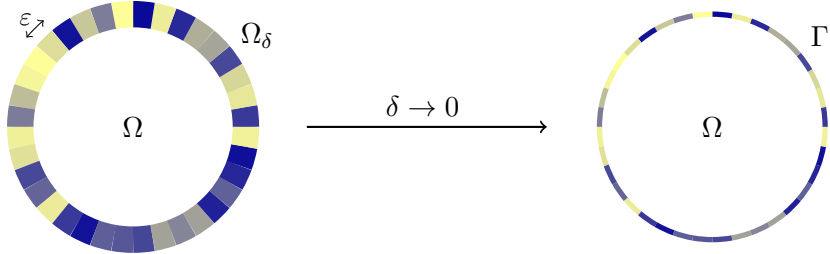
This leads to a coupled bulk-bulk problem of the form 
\begin{subequations}
\label{eq:coupledSys}
\begin{align}
	\dot  u - \kappa\, \Delta u &= f \qquad\text{in } \Omega, \label{eq:coupledSys:a} \\
	\dot w - \nabla( a_\eps \nabla w) &= 0 \qquad \text{on } \Omega_{\delta}, \label{eq:coupledSys:b} \\
	w - u &= 0 \qquad \text{on } \Gamma, \label{eq:coupledSys:c} \\
	\kappa\, \partial_\nu u -\delta^{-1}a_\eps\partial_\nu w&=0 \qquad \text{on }\Gamma, \label{eq:coupledSys:d}\\
	a_\eps \partial_\nu w &=0\qquad \text{on }\partial\Omega_\delta\setminus \Gamma \label{eq:coupledSys:e}
\end{align}
\end{subequations}
with initial conditions for $u$ and $w$. 
Here, $\nu$ denotes the normal outer vector on the boundary. 
Assuming~$\delta$ to be small, we consider the limit $\delta\to 0$ in order to reduce the surrounding domain~$\Omega_{\delta}$ to a boundary layer. 
This means that the original bulk-bulk coupling is replaced by a bulk-surface coupling, which can be considered as a PDE with non-standard boundary conditions. 
More precisely, this leads to a heat equation with dynamic (and heterogeneous) boundary conditions, namely 
\begin{subequations}
\label{eq:dynBC}
\begin{align}
	\dot  u - \kappa\, \Delta u &= f \qquad\text{in } \Omega, \label{eq:dynBC:a} \\
	\dot u - \nabla_\Gamma\cdot( a_\eps \nabla_\Gamma u) + \kappa\,\partial_\nu u &= 0 \qquad \text{on } \Gamma. \label{eq:dynBC:b}
\end{align}
\end{subequations}
Note that by~$\nabla_\Gamma$ we denote the gradient in tangential direction of the boundary~$\Gamma=\partial\Omega$. 
In the special case $a_\eps\equiv 1$ the differential operator simplifies to the Laplace-Beltrami operator, see~\cite[Ch.~16.1]{GilT01}. 
In the general case with~$a_\eps\in L^\infty(\Gamma)$, $a_\eps \ge \alpha > 0$, the corresponding operator~$\calA_\eps\colon H^1(\Gamma) \to H^1(\Gamma)^*$ still satisfies a G\aa rding inequality, namely 
\[
  \langle \calA_\eps p, p\rangle 
  = \int_\Gamma a_\eps \nabla_\Gamma p\cdot \nabla_\Gamma p\dx
  \ge \alpha\, \|\nabla_\Gamma p\|^2_{L^2(\Gamma)}
  = \alpha\, \| p\|^2_{H^1(\Gamma)} - \alpha\, \| p\|^2_{L^2(\Gamma)}.
\]
The following result is devoted to the limiting process. 
\begin{theorem}\label{thm:limit}
Assume the boundary~$\Gamma$ to be smooth. 
Then, system~\eqref{eq:dynBC} is the limit of~\eqref{eq:coupledSys} as~$\delta\to 0$.
\end{theorem}
\begin{proof}
This follows with the arguments of~\cite{Lie13} in the special case of $\alpha_r=\beta_E=1$ and $W\equiv 0$. The only difference is that therein the coefficient $a_\eps$ is assumed to be constant. 
We briefly outline the main steps and modifications and refer to~\cite{Lie13} for more details. 
It is still possible to reformulate the bulk-bulk problem \eqref{eq:coupledSys} as an energy balance. Due to the uniform positive lower bound on $a_\eps$ the a priori estimates for the solution derived in~\cite{Lie13} still hold true.
Employing them, the convergence of the dynamic part in the energy balance formulation follows in the same way as before because we have made no alterations concerning the time derivatives.
Thus, we only have to show the {\em Mosco convergence} of the energy functional, which corresponds to the elliptic parts in \eqref{eq:coupledSys}. Since this requires some further notation, we postpone the detailed proof to~Appendix~\ref{app:Mosco}.
\end{proof}
\begin{remark}
Since the limiting process is only concerned with the equation in $\Omega_\delta$, one can replace the heat equation in the interior domain~$\Omega$ by a more involved or even nonlinear parabolic problem. 
Furthermore, it is possible to include in~\eqref{eq:coupledSys:b} a reacting term $\alpha u$ as well as a (sufficiently smooth) inhomogeneity $g\in L^2(\Omega_\delta)$, which only varies in tangential direction. 
\end{remark}
Throughout this paper, equation~\eqref{eq:dynBC} serves as a model problem for a parabolic system with dynamic boundary conditions including heterogeneities. Additionally, we allow an inhomogeneity $g\in L^2(\Gamma)$ in~\eqref{eq:dynBC:b} and assume~$\Gamma$ to be polygonal/polyhedral and thus, only Lipschitz continuous. 
The latter may be given as an approximation of a smooth domain, meaning that the obtained finite element approximation includes an error coming from the discrepancy of the boundaries. Such situations can be analyzed with the help of a {\em lift operator}, cf.~\cite{Dzi88} or~\cite[Sect.~4.2]{DziE13}, but are not in the focus of this research. 
%
%
\subsection{Weak formulation}\label{sect:model:weak}
In the remainder of this paper, the computational domain~$\Omega\subseteq \R^d$ is assumed to be a Lipschitz domain with a polygonal/polyhedral boundary~$\Gamma$, on which the dynamic boundary conditions are defined.
Further we assume right-hand sides~$f\colon [0,T] \to L^2(\Omega)$ and~$g\colon [0,T] \to L^2(\Gamma)$. 
The weak formulation presented in~\cite{KovL17} reads
\begin{align}
\label{eq:weakKonL}
  m(\dot u, v) + a(u, v) 
  = (f, v)_\Omega + (g, v)_\Gamma
\end{align}
with~$(\,\cdot\,, \cdot\,)_\Omega$ and~$(\,\cdot\,, \cdot\,)_\Gamma$ denoting the $L^2$-inner products on~$\Omega$ and~$\Gamma$, respectively, and with the bilinear forms
\begin{align*}
  m(u, v)
  \coloneqq \int_{\Omega} u\, v\dx + \int_{\Gamma} u\, v\dx, \quad
  a(u, v) 
  \coloneqq \int_{\Omega} \kappa\, \nabla u\cdot \nabla v\dx + \int_{\Gamma} a_\eps \nabla_\Gamma u\cdot \nabla_\Gamma v\dx. 
\end{align*}
The corresponding trial and test space is given by $\Vkl \coloneqq \{ v\in H^1(\Omega)\ |\ v|_\Gamma \in H^1(\Gamma) \}$. 
Thus, the two bilinear forms include boundary integrals and the space~$\Vkl$ requires a trace in $H^1(\Gamma)$. 
We will see in Section~\ref{sect:disc:naive} that this approach is not suitable for the considered situation with a heterogeneity on the boundary. 

In order to allow different discretizations in the bulk and on the boundary later on, we derive an alternative weak formulation. 
Here we follow~\cite{Alt19} and consider~\eqref{eq:dynBC} as a coupled system, which interprets $u$ and $u|_\Gamma$ as two separate variables.  
For this, we introduce~$p\coloneqq u|_\Gamma$ as a new variable, which turns the dynamic boundary condition~\eqref{eq:dynBC:b} into~$\dot p - \nabla_\Gamma\cdot( a_\eps \nabla_\Gamma p) + \kappa\,\partial_\nu u = g$ on~$\Gamma$. 
As ansatz spaces we define
\[
  \V \coloneqq H^1(\Omega), \qquad
  \Q \coloneqq H^1(\Gamma)
\]
for~$u$ and~$p$, respectively. 
Note that we need a trace in $H^1(\Gamma)$ due to the generalized Laplace-Beltrami operator but do not include this into the space~$\V$. 

Considering test functions $v\in\V$ for~\eqref{eq:dynBC:a} and $q\in \Q$ for~\eqref{eq:dynBC:b}, integration by parts yields 
\begin{subequations}
	\label{eq:weakForm}
	\begin{align}
	\int_{\Omega} \dot u\, v \dx + \int_\Omega \kappa\, \nabla u\cdot \nabla v \dx - \int_\Gamma \kappa\, \partial_\nu u\, v \dx 
	&= \int_\Omega f\, v\dx, \label{eq:weakForm:a}\\
	\int_\Gamma \dot p\, q \dx + \int_\Gamma a_\eps \nabla_\Gamma p\cdot \nabla_\Gamma q\dx + \int_\Gamma \kappa\,\partial_\nu u\, q\dx 
	&= \int_\Gamma g\, q\dx. \label{eq:weakForm:b}
	\end{align}
\end{subequations}
For the differential operators we introduce~$\calK\colon \V \to \V^*$ as the weak Laplacian (weighted by~$\kappa$) and~$\calA_\eps\colon \Q \to \Q^*$ defined by $\langle \calA_\eps p, q\rangle \coloneqq \int_\Gamma a_\eps \nabla_\Gamma p\cdot \nabla_\Gamma q\dx$. 
Moreover, we implement the connection of~$u|_\Gamma$ and $p$ in form of a constraint to the system equations. 
With~$\M \coloneqq H^{-1/2}(\Gamma)$ we define the constraint operator~$\calB\colon (\V \times \Q) \to \M^* = H^{1/2}(\Gamma)$ by 
\[
\calB(u,p) 
\coloneqq p - u|_\Gamma.
\]
We emphasize that this operator satisfies an inf-sup condition, see~\cite[Lem.~5]{Alt19}. 
Using the Lagrangian method to enforce the constraint, we introduce an additional unknown, namely the Lagrange multiplier~$\lambda\colon [0,T] \to \M$, which leads to the PDAE formulation 
\begin{subequations}
	\label{eq:PDAE}
	\begin{align}
	\begin{bmatrix} \dot u \\ \dot p  \end{bmatrix}
	+ \begin{bmatrix} \calK &  \\  & \calA_\eps \end{bmatrix}
	\begin{bmatrix} u \\ p  \end{bmatrix}
	+ \calB^*\lambda 
	&= \begin{bmatrix} f \\ g \end{bmatrix} \qquad \text{in } \V^*\times \Q^*, \label{eq:PDAE:a}\\
	\calB(u,p) \hspace{3.1em} &= \phantom{[]} 0\hspace{2.8em} \text{in } \M^*. \label{eq:PDAE:b}
	\end{align}
\end{subequations}
Note that this includes operator matrices and that a test function~$(v,0) \in \V\times\Q$ applied to~\eqref{eq:PDAE:a} equals equation~\eqref{eq:weakForm:a}, where the appearance of the normal derivative of~$u$ has been replaced by the Lagrange multiplier. 
Accordingly, the test function~$(0, q) \in \V\times\Q$ yields~\eqref{eq:weakForm:b} and shows that the PDAE~\eqref{eq:PDAE} is equivalent to the weak formulation~\eqref{eq:weakForm}. In particular, one can show under sufficient regularity assumptions that the solutions coincide with~$\lambda=\kappa\,\partial_\nu u$.

For later use, we define the bilinear forms associated to the differential operators~$\calK$ and~$\calA_\eps$, namely 
\[
  \frakk(u,v) 
  \coloneqq \langle\calK u, v\rangle,\qquad 
  \fraka(p,q) 
  \coloneqq \langle\calA_\eps p, q\rangle. 
\]
Recall that~$\alpha>0$ denotes the lower bound of the diffusion coefficient~$a_\eps$ and that the bilinear form~$\fraka$ only satisfies a G\aa rding inequality in~$\Q$. 
Because of this, we introduce the elliptic bilinear form~$\frakaTilde\colon \Q\times\Q\to\R$ by
\[
  \frakaTilde(p, q)
  \coloneqq \fraka(p ,q) + (\alpha\,p ,q)_\Gamma,  
\]
which satisfies~$\frakaTilde(p, p) \ge \alpha\, \|p\|^2_\Q$. 
Introducing the bilinear form~$\frakb\colon \M^* \times\M\to\R$ by~$\frakb(q,\lambda) \coloneqq \langle q,\lambda\rangle_\Gamma$, we have 
\[
  \frakb(q-v|_\Gamma,\lambda) 
  = \langle q-v|_\Gamma,\lambda\rangle_\Gamma 
  = \langle\calB(v,q), \lambda\rangle. 
\]
As a consequence, we can rewrite system~\eqref{eq:PDAE} in the form 
\begin{subequations}
\label{eq:PDAE:bilinear}
\begin{align}
	(\dot u, v)_\Omega + (\dot p, q)_\Gamma 
	+ \frakk(u,v) + \fraka(p,q)
	+ \frakb(q-v|_\Gamma, \lambda)
	&= (f, v)_\Omega + ( g, q)_\Gamma, \\
	\frakb(p-u|_\Gamma, \mu) &= 0
\end{align}
\end{subequations}
with test functions~$v\in\V$, $q\in\Q$, and~$\mu\in\M$. 
%
%
\subsection{Regularity of inhomogeneous solutions}\label{sect:model:regularity}
Standard discretization schemes for homogeneous Dirichlet boundary problems defined on convex Lipschitz domains usually assume $H^2$-regularity of the solution leading to the well-known optimal convergence rates.  
Prescribed Dirichlet data may already change the regularity of the solution, since a trace in~$H^{1/2}(\Gamma)$ only implies a solution in~$H^1(\Omega)$. 

For dynamic boundary conditions on smooth domains including the Laplace-Beltrami operator, a standard regularity assumption is~$u(t)\in H^2(\Omega)$ with $u(t)|_\Gamma \in H^2(\Gamma)$. In this situation, linear finite elements yield the full second-order convergence in the $L^2$-norm, cf.~\cite[Th.~3.6]{KovL17}. 
Recall that we only consider Lipschitz domains and that we additionally have a heterogeneous diffusion coefficient~$a_\eps$ such that $H^2$-regularity is not to be expected on the boundary. Further, the coupling of bulk and surface dynamics do not allow a simple decomposition of the solution as in the pure Dirichlet case. 
The quite general assumptions on the computational domain and the coefficient~$a_\eps$ only ensure a solution~$p$ with values in~$\Q = H^1(\Gamma)$. By standard results on the trace operator, see e.g.~\cite[Ch.~2.6]{SauS11}, we cannot expect more regularity than~$u(t)\in H^{3/2}(\Omega)$ in the interior. This also fits to the numerical observations in Section~\ref{sect:numerics:exp1}.

For polygonal convex domains in two space dimensions, i.e., $d=2$, we can expect~$u(t)\in H^2(\Omega)$ if~$p(t)\in H^{3/2}(\Gamma_i)$ for each edge $\Gamma_i \subseteq \Gamma$ and $p(t)\in C(\partial\Omega)$, i.e., $p$ is continuous at the vertices of the boundary, cf.~\cite[Thm.~5.1.2.4]{Gri85}. 
For~$d=3$ we know from~\cite[Cor.~5.5.2, Rem.~5.5.3]{Moi11} that in a convex Lipschitz polyhedron we have~$u(t)\in H^{s+1/2}(\Omega)$ if~$p(t)\in H^1(\Gamma)$, $p(t)\in H^s(\Gamma_i)$ for each face~$\Gamma_i\subseteq\Gamma$, and~$\Delta u\in L^2(\Omega)$ for~$1<s<3/2$. 
We emphasize that the ``border cases''~$s=1$ and~$s=3/2$ are excluded, so that on the one hand, $p$ needs to be a little more regular than simply $H^1(\Gamma)$ and, on the other hand, one cannot conclude full $H^2$-regularity of~$u$ with this argument. 
\begin{example}
In the numerical examples of Section~\ref{sect:numerics} we will consider~$d=2$ and the smooth but highly oscillatory coefficient
\[
a_\eps^\text{sm}(x) 
\coloneqq \frac{1}{ 2 + \cos(2\pi x \eps^{-1}) }.
\]
For the corresponding stationary problem~$-\nabla_\Gamma\cdot( a_\eps \nabla_\Gamma p) = g$ this results in a solution, where higher-order norms of $p$ are expected to scale with a negative power of $\eps$, e.g., $\|p\|_{H^{s}(\Gamma)}$  scales like~$\eps^{1-s}$ for integers $s\geq 1$, as discussed in~\cite{PetS12}. 
Moreover, if~$u$ solves the Poisson equation with boundary data~$p$, then~\cite[Cor.~5.5.2, Rem.~5.5.3]{Moi11} provides the stability result
\[
  \|u\|_{H^{s}(\Omega)}
  \leq C\, \big(\|p\|_{H^{s-1/2}(\Gamma)}+\|\Delta u\|_{L^2(\Omega)} \big).
\]
Thus, oscillatory boundary conditions lead to~$\eps$-dependent bounds of~$u$ for~$s>3/2$. This motivates the application of multiscale methods, which enable~$\eps$-independent convergence rates. 
\end{example}
%
%
\section{Spatial Discretization}\label{sect:disc}
The heterogeneous nature of the boundary calls for a multiscale method on~$\Gamma$. 
Because of the very general structure of the diffusion coefficient~$a_\eps$, which does not assume periodicity or any separation of scales, we consider the LOD.  
On the other hand, we have a homogeneous problem in the bulk such that standard finite elements yield satisfactory results. 
We first follow the naive approach, which requires a very high mesh resolution also of the homogeneous domain. 
In order to allow different discretizations in the bulk and on the boundary, we design mixed finite element schemes based on the alternative formulation~\eqref{eq:PDAE}. 
This then enables efficient numerical schemes, which combine coarse grids on~$\Omega$ with multiscale methods on~$\Gamma$. 

Throughout the paper we write~$a \lesssim b$ to indicate that there exists a generic constant~$C$, independent of spatial and temporal discretization parameters, such that~$a \leq C b$. 
%
%
\subsection{The naive finite element approach}\label{sect:disc:naive}
A straight-forward finite element approach considers the weak formulation~\eqref{eq:weakKonL} together with a uniform triangulation. 
The corresponding Galerkin approximation is given by~$u_h\colon [0,T] \to \Vhkl\subseteq \Vkl = \{ v\in H^1(\Omega)\ |\ v|_\Gamma \in H^1(\Gamma) \}$  and solves the semi-discrete system 
\begin{align}
\label{eq:KonL:semidiscrete} 
  m(\dot u_h, v_h) + a(u_h, v_h) 
  = (f, v_h)_\Omega + (g, v_h)_\Gamma
\end{align}
for all $v_h\in \Vhkl$ and some initial condition for $u_h(0)$. 
Let~$\calT$ be a uniform triangulation of the computational domain~$\Omega$ with mesh size~$h$ and~$\Vhkl \coloneqq \calP_1(\calT)$ the space of Lagrange finite elements, i.e., piecewise affine and globally continuous functions. 
In this setting, it is well-known that a high resolution, i.e., $h\lesssim \eps$, is necessary to capture the microscopic effects~\cite{Pet16}. 
This is due to the fact that~$\nabla_\Gamma v_h$ is piecewise constant on the boundary and thus, only the arithmetic mean of~$a_\eps$ enters the semi-discrete equations for coarse~$h$. 
This, however, does not reflect the correct microscopic behavior and leads to an extended pre-asymptotic phase in the approximation. 
Thus, the highly oscillating diffusion coefficient on the boundary needs to be compensated by a very small mesh size~$h$. 

Of course, this illustrates only the worst case and adaptive finite elements~\cite[Ch.~9]{BreS08} or general multiscale methods such as the LOD~\cite{MalP14} can be used to overcome these difficulties. 
Both approaches, however, have in common that the heterogeneity on the boundary affects the mesh in the interior of the domain. 
The method introduced in this paper follows a different paradigm, namely the reformulation of the problem, which decouples the dynamics in the bulk and on the boundary. 
With this strategy it is sufficient to apply standard finite element schemes in the bulk and multiscale methods only on the surface and thus, on a lower-dimensional domain. 
%
%
\subsection{Mixed finite elements} 
As an alternative, we now consider discretizations of the PDAE system~\eqref{eq:PDAE}. 
This leads to so-called {\em mixed methods}~\cite[Ch.~III.4]{Bra07} with two different discretization schemes for~$u$ and~$p$. 
We consider conforming finite element discretizations based on finite-dimensional spaces 
\[
  \Vh \subseteq \V, \qquad
  \Qh \subseteq \Q, \qquad
  \Mh \subseteq \M
\]
of dimension~$n_u$, $n_p$, and $n_\lambda$, respectively.  
The spaces $\Vh$ and $\Mh$ will be discrete spaces consisting of piecewise polynomials based on a triangulation~$\calTO$ of $\Omega$. 
For simplicity we assume that~$\calTO$ is a uniform mesh with mesh size~$\HO$. 
On the other hand, $\Qh$ is defined w.r.t.~a mesh~$\calTG$ of $\Gamma$, which may be chosen independently of~$\calTO$. 

In any case, the Galerkin discretization of~\eqref{eq:PDAE} reads as follows: 
Find $\uh\colon [0,T] \to \R^{\dimu}$, $\ph\colon [0,T] \to \R^\dimp$, and~$\lh\colon [0,T] \to \R^\diml$ such that 
\begin{align*}
	\begin{bmatrix} M_\Omega &  \\  & M_\Gamma \end{bmatrix}
	\begin{bmatrix} \duh \\ \dph  \end{bmatrix}
	+ \begin{bmatrix} K &  \\  & A_\eps \end{bmatrix}
	\begin{bmatrix} \uh \\ \ph \end{bmatrix}
	+ B^T \lh 
	&= \begin{bmatrix} b_\Omega \\ b_\Gamma \end{bmatrix}, \\ 
	B \begin{bmatrix} \uh \\ \ph \end{bmatrix} \hspace{1.53cm}
	&= 0 
\end{align*}
for almost all times $t\in [0,T]$. 
Here, $M_\Omega\in\R^{\dimu, \dimu}$ and $M_\Gamma\in\R^{\dimp,\dimp}$ denote the mass matrices corresponding to an appropriate basis of~$\Vh$ and $\Qh$, respectively. 
The stiffness matrices $K\in \R^{\dimu,\dimu}$ and $A_\eps\in \R^{\dimp,\dimp}$ are the discrete versions of the differential operators~$\calK$ and $\calA_\eps$. 
Finally, $B\in \R^{\diml,\dimu+\dimp}$ is the discrete analogue of the constraint operator~$\calB$ and $b_\Omega$, $b_\Gamma$ correspond to the right-hand sides~$f$ and $g$, respectively. 

Although the discretization spaces~$\Vh$, $\Qh$, and $\Mh$ can be chosen independently on first sight, they need to be suitable in the sense of an inf-sup condition. 
In the following, we need to guarantee that the discrete spaces satisfy 
\begin{align}
\label{eq:discInfSup}
  \adjustlimits\inf_{\mu_H\in \Mh} \sup_{v_H\in \Vh, q_H\in \Qh} 
  \frac{|\langle \calB(v_H, q_H), \mu_H \rangle|}{\| (v_H, q_H) \|_{\V\times\Q} \|\mu_H\|_\M}
  \ge \beta > 0
\end{align}
with a constant~$\beta$, independent of the mesh sizes. 
Here, the corresponding norm is defined through
\[ 
  \| (v_H, q_H) \|_{\V\times\Q}
  \coloneqq \big( \| v_H \|^2_{\V} + \| q_H \|^2_{\Q} \big)^{1/2}.
\]
Note that such a condition automatically implies the full rank property of the (discrete) constraint matrix~$B$.

The remaining task is to find suitable spaces~$\Vh$, $\Qh$, and~$\Mh$, which have good approximation properties and are stable in the sense of~\eqref{eq:discInfSup}. 
For this, we collect a number of standard finite element spaces, which will be used in the following. 
First, we introduce the standard piecewise polynomial spaces for triangulations~$\calTO$ into triangles (tetrahedra for $d=3$), namely
\begin{align*}
  \calP_k(\calTO) 
  &\coloneqq \big\{ v\in\V\ \big|\ v|_T \text{ is a polynomial of degree} \le k \text{ for all } T\in\calTO \big\}
\end{align*}
for~$k\ge 1$. 
Note that these spaces are of conforming type by definition and globally continuous.  
Similarly, we may define piecewise polynomial spaces for partitions into quadrilaterals (cubes for $d=3$), see~\cite[Ch.~3.5]{BreS08}. 
For the partition of the boundary, we define accordingly 
\begin{align*}
  \calP_k(\calTG) 
  &\coloneqq \big\{ v\in\Q\ \big|\ v|_T \text{ is a polynomial of degree} \le k \text{ for all } T\in\calTG \big\}
\end{align*}
for~$k\ge 1$. 
These function spaces are again globally continuous. On the boundary we also consider the discontinuous spaces 
\[
  \calP^\text{d}_\ell(\calTG) 
  \coloneqq \big\{ v\in L^2(\Gamma)\ \big|\ v|_T \text{ is a polynomial of degree} \le \ell \text{ for all } T\in\calTG \big\} 
\]
for~$\ell\ge 0$. 
Note that this defines a conforming subspace for $\M$. Finally, we introduce the space of edge/face-bubbles by
\begin{align*}
  \calE_\ell(\calTO) 
  &\coloneqq \big\{ v\cdot\psi_E\ \big|\ v|_T \text{ is a polynomial of degree} \le \ell \text{ for all } T\in\calTO, \\
  &\hspace{4.4cm}\psi_E \text{ is an edge/face-bubble for } E\subseteq\Gamma \big\} \subseteq \V. 
\end{align*}
Here, an edge-bubble~$\psi_E$ (face-bubble for $d=3$) equals the scaled product of the two (three for $d=3$) corresponding nodal basis functions~\cite[Ch.~1]{Ver96}. 
%
%
\subsection{Inf-sup stable discretizations}\label{sect:disc:infsup}
In this subsection, we present two classes of discretization schemes which are stable in the sense of the inf-sup condition~\eqref{eq:discInfSup}. 
They have in common that the stability is independent of the choice of~$\Qh$. 
\begin{proposition}\label{prop:InfsupBubbleTriang}
The conforming finite element spaces
\[
  \Vh \coloneqq \calP_k(\calTO) \oplus \calE_\ell(\calTO) \subseteq \V, \quad
  \Qh \subseteq \Q, \quad
  \Mh \coloneqq \calP^\text{d}_\ell(\calTO|_\Gamma) \subseteq \M
\]
satisfy a discrete inf-sup condition~\eqref{eq:discInfSup} for all parameters~$k\ge 1$, $\ell\ge 0$ and arbitrary~$\Qh$.  
\end{proposition}
\begin{proof}
For an arbitrary $\mu_H\in \Mh \subseteq L^2(\Gamma)$ we set $q_H = 0$ and note that 
\[
  \sup_{v_H\in \Vh, q_H\in \Qh} 
  \frac{|\langle \calB(v_H, q_H), \mu_H \rangle|}{\| (v_H, q_H) \|_{\V\times\Q} \|\mu_H\|_\M}
  \ge \sup_{v_H\in \Vh} 
  \frac{(v_H|_\Gamma, \mu_H)_{\Gamma}}{\| v_H \|_{H^1(\Omega)} \|\mu_H\|_{H^{-1/2}(\Gamma)}}. 
\]
For this term, the inf-sup stability shown in~\cite[Th.~2.3.7]{Lip04} can be applied. 
More details can be found in~\cite[Ch.~4.1]{Wie19}. 
\end{proof}
Note that the choice of $\Qh$ does not influence the stability of the scheme, which allows to implement special multiscale finite element spaces at this point. This will be discussed in Section~\ref{sect:disc:LOD} below. 
\begin{remark}
The result of Proposition~\ref{prop:InfsupBubbleTriang} also holds true on quadrilateral meshes if the discrete space~$\calP_k(\calTO)$ is replaced by the corresponding space~$\calQ_k(\calTO)$ of piecewise polynomials with partial degree~$k$. 
\end{remark}
\begin{remark}
In the schemes considered in Proposition~\ref{prop:InfsupBubbleTriang} the Lagrange multipliers are defined on the mesh given by $\calTO|_\Gamma$ and the stabilization occurs with the help of bubble functions. We emphasize that a stabilization using $p$ is not as straightforward, since the norms in the inf-sup condition~\eqref{eq:discInfSup} do not match. This is due to the fact that we do not include $u|_\Gamma \in H^1(\Gamma)$ into the continuous model~\eqref{eq:PDAE}. 
\end{remark}	
The following result shows that stable schemes also exist without the need of bubble functions. 
\begin{proposition}
\label{prop:P1P1}
The conforming finite element spaces
\[
  \Vh\coloneqq \calP_1(\calTO)\subseteq \V, \quad
  \Qh\subseteq\Q, \quad
  \Mh\coloneqq \calP_1(\calTO|_\Gamma)\subseteq \M
\]
satisfy a discrete inf-sup condition~\eqref{eq:discInfSup} for arbitrary~$\Qh$.
\end{proposition}
\begin{proof}
For an arbitrary $\mu_H\in \Mh$ with $\|\mu_H\|_{H^{-1/2}(\Gamma)}=1$, let $\tilde{\mu}\in H^{1/2}(\Gamma)$ be its Riesz representative with $\|\tilde \mu\|_{H^{1/2}(\Gamma)}=1$. 
Let $\tilde{\mu}_H\in \calP_1(\calTO|_\Gamma)$ be the $L^2$-projection of $\tilde{\mu}$. 
We set $q_H=0$ and $v_H=-E_H^0\tilde{\mu}_H$ with the extension operator $E_H^0$ from \cite[Lem.~3.1]{HipJM15}. Note that this is \emph{not} the trivial extension by zero on the interior nodes.
We then obtain
\begin{align*}
\sup_{v_H\in \Vh, q_H\in \Qh} 
\frac{\langle \calB(v_H, q_H), \mu_H \rangle}{\| (v_H, q_H) \|_{\V\times\Q}}
&\ge \frac{(-v_H|_\Gamma, \mu_H)_{\Gamma}}{\| v_H \|_{\V}}\\
&= \frac{(\tilde{\mu}_H, \mu_H)_{\Gamma}}{\| v_H \|_{\V}}
=\frac{(\tilde{\mu}, \mu_H)_{\Gamma}}{\| v_H \|_{\V}}
=\frac{\|\tilde{\mu}\|^2_{H^{1/2}(\Gamma)}}{\| v_H \|_{\V}},
\end{align*}
where we used the property of the $L^2$-projection.
According to \cite[Lem.~3.1]{HipJM15}, we have $\|v_H\|_{H^1(\Omega)}\lesssim \|\tilde{\mu}_H\|_{H^{1/2}(\Gamma)}$. 
Recall that the $L^2$-projection is stable in $H^1(\Gamma)$, and thus, also in $H^{1/2}(\Gamma)$ (see, e.g., \cite{BraPS02}) so that $\|\tilde{\mu}_H\|_{H^{1/2}(\Gamma)}\lesssim \|\tilde{\mu}\|_{H^{1/2}(\Gamma)}$.
All in all, we deduce
\[
\sup_{v_H\in \Vh, q_H\in \Qh} 
\frac{\langle \calB(v_H, q_H), \mu_H \rangle}{\| (v_H, q_H) \|_{\V\times\Q}}
\gtrsim \|\tilde{\mu}\|_{H^{1/2}(\Gamma)}=1.\qedhere
\]
\end{proof}
\begin{remark}
The above result can be generalized to $\Vh\coloneqq\calP_k(\calTO)$ and $\Mh\coloneqq\calP_\ell(\calTO|_\Gamma)$ for $0<\ell\leq k$ provided that (i) an extension operator $E_H:\calP_k(\calTO)\to \calP_k(\calTO|_\Gamma)$ in the spirit of \cite[Lem.~3.1]{HipJM15} exists and that (ii) the $L^2$-projection onto $\calP_k(\calTO|_\Gamma)$ is stable in $H^{1/2}(\Gamma)$. 
\end{remark}
\begin{remark} 
\label{rem:Q1}
Also the result of Proposition~\ref{prop:P1P1} can be extended to quadrilateral meshes, i.e., we may replace~$\calP_1(\calTO)$ and~$\calP_1(\calTO|_\Gamma)$ by the corresponding spaces of piecewise polynomials of partial degree~$1$. 
For $d=2$ this only requires the bisection of all quadrilaterals into triangles. Then, we can follow the proof of Proposition~\ref{prop:P1P1} using a~$\V$-continuous mapping from the corresponding~$\calP_1$ space to~$\calQ_1(\calTO)$, i.e., to the space of the original partition.  
This conversion is more involved for~$d=3$, where the bulk and surface partition both need a transformation with particular continuity assumptions. 
\end{remark}
Similar to Proposition~\ref{prop:InfsupBubbleTriang}, the inf-sup stability in Proposition~\ref{prop:P1P1} is independent of the choice of $\Qh$, which allows to insert an LOD space at this point.
Moreover, we emphasize that, in the case of a one-dimensional boundary, the ``over-regular'' discretization of $\calM$ by the $H^1$-conforming space $\calP_1(\calTG)$ does not lead to more degrees of freedom than the choice $\calP^\text{d}_0(\calTG)$. 
%
%
\subsection{LOD function spaces}\label{sect:disc:LOD}
Let $\calT_{\Gamma,h}$ and $\calT_{\Gamma,H}$ be two uniform meshes on $\Gamma$ with mesh sizes $h_\Gamma$ and $H_\Gamma$, respectively. We assume $h_\Gamma$ to be fine in the sense that it resolves the oscillations and discontinuities of $a_\eps$, whereas $\calT_{\Gamma,H}$ is assumed to be coarse in the sense that it is the restriction of  $\calT_\Omega$ to the boundary and in general does not resolve $a_\eps$.
Moreover, we assume that $\calT_{\Gamma,h}$ is a refinement of $\calT_{\Gamma,H}$.
To obtain faithful approximations with the standard finite element method, we need to utilize spaces $\calP_k(\calT_{\Gamma,h})$, which lead to a large number of degrees of freedom.
Instead, we introduce the Localized Orthogonal Decomposition, which modifies the space $\calP_1(\calT_{\Gamma,H})$ such that it yields satisfactory approximations.
For the construction, we consider the stable and surjective Cl{\'e}ment-type (quasi-) interpolation operator~$I_H\colon \calP_1(\calT_{\Gamma,h})\to \calP_1(\calT_{\Gamma,H})$ introduced in~\cite{CarV99}, namely 
\[
  I_H v
  \coloneqq\sum_{z\in\calN_{\Gamma, H}}\frac{(v, \phi_z)_\Gamma}{(1,\phi_z)_\Gamma}\,\,\phi_z. 
\]
Here, $\calN_{\Gamma, H}$ denotes the set of vertices of $\calT_{\Gamma, H}$ and $\phi_z\in \calP_1(\calT_{\Gamma, H})$ is the standard nodal basis function (``hat function'') associated with the vertex $z$.
We denote by $W_h$ the kernel of this interpolation operator.
Moreover, we note the following properties of $I_H$ that we will use in the error analysis.
For any $T\in \calT_{\Gamma,H}$ and~$q\in \calP_1(\calT_{\Gamma,h})$ the operator~$I_H$  satisfies the estimate 
\begin{equation}\label{eq:IH:stabapprox}
  H_\Gamma^{-1}\|q-I_H q\|_{L^2(T)}+\|\nabla I_H q\|_{L^2(T)}
  \lesssim \|\nabla q\|_{L^2(U(T))},
\end{equation}
where $U(T)$ denotes all neighboring elements of $T$, i.e., $U(T)=\{T^\prime \in \calT_{\Gamma,H}\, |\, T^\prime\cap T\neq\emptyset\}$.
Denoting by $\Pi_H$ the global $L^2(\Gamma)$-projection onto $\calP_1(\calT_{\Gamma,H})$, it holds that $I_H=I_H\circ \Pi_H$ and  $W_h=\ker(I_H)=\ker(\Pi_H|_{\calP_1(\calT_{\Gamma,h})}) \subseteq \calP_1(\calT_{\Gamma,h})$, see \cite{MalP15}.
Hence, $\calP_1(\calT_{\Gamma,h})=W_h\oplus \calP_1(\calT_{\Gamma,H})$
with
\begin{equation}\label{eq:IH:L2ortho1}
(\calP_1(\calT_{\Gamma,H}), W_h)_{\Gamma}=0,
\end{equation}
see \cite{MalP15}.

In the next step, we orthogonalize this splitting with respect to the operator $\calA_\eps$. 
For this, recall the definition of the elliptic bilinear form~$\frakaTilde$ introduced in Section~\ref{sect:model:weak}. 
We define the corrector Green's operator $\calG\colon \calP_1(\calT_{\Gamma,H})\to W_h$ via
\begin{equation}\label{eq:corrector}
  \frakaTilde(\calG q_H, w)
  = \frakaTilde(q_H, w)\qquad \text{for all}\quad w\in W_h.
\end{equation}
Note that \eqref{eq:corrector} is well-posed by the Lax-Milgram Theorem.  
The corrector Green's operator can be decomposed into~$\calG=\sum_{T\in \calT_{\Gamma,H}}\calG_T$, where $\calG_T$ solves 
\begin{equation}\label{eq:corrector-element}
  \frakaTilde(\calG_T q_H, w) 
  = \frakaTilde_T(q_H, w)\qquad \text{for all}\quad w\in W_h
\end{equation}
with $\frakaTilde_T$ being the restriction of $\frakaTilde$ to an element $T\subseteq \Gamma$, i.e., $\tilde\fraka_T(q_H, w) = \int_T a_\eps \nabla_\Gamma q_H\cdot \nabla_\Gamma w+\alpha q_H w\dx$. 
Since the computation of $\calG_T$ requires the solution of global fine-scale problems in general, we need to introduce localized approximations $\calG_{T,m}$ of $\calG_T$ and $\calG_m$ of $\calG$, respectively.
Let the $m$-th layer patch $U_m(T)$ be defined inductively as
\[ 
  U_m(T) \coloneqq U(U_{m-1}(T)), \qquad 
  U_0(T) \coloneqq T.
\]
The localized or truncated element corrector $\calG_{T,m}\colon \calP_1(\calT_{\Gamma,H})\to W_h\cap H^1_0(U_m(T))$ is now defined via
\begin{equation}\label{eq:corrector-element-local}
 \frakaTilde_{U_m(T)}(\calG_{T, m} q_H, w)
  =\frakaTilde_T(q_H, w)\qquad \text{for all}\quad w\in W_h\cap H^1_0(U_m(T))
\end{equation}
and we set $\calG_m\coloneqq\sum_{T\in \calT_{H, \Gamma}}\calG_{T, m}$. 
\begin{remark}
We define the correctors $\calG_T$ and $\calG_{T,m}$ with respect to the modified bilinear form~$\frakaTilde$, since it simplifies the analysis in Section~\ref{sect:errAnalysis}. 
However, a definition via~$\fraka$ is equally possible and leads to the same convergence results.
\end{remark}
The error between $\calG$ and $\calG_m$ decays exponentially with $m$ in the $H^1(\Gamma)$-norm as specified in the next lemma. For a proof we refer to, e.g., ~\cite{MalP14}.
\begin{lemma}
There exists a constant $0<\gamma<1$, which is independent of $h_\Gamma$, $H_\Gamma$, and $m$, such that for any $q_H\in Q_H$ it holds that
\begin{equation}
\label{eq:corrector:decay}
\|(\calG-\calG_m)q_H\|_{\Q}
\lesssim m^{(d-1)/2}\, \gamma^m\,\|q_H\|_{\Q}.
\end{equation}
\end{lemma}
\begin{remark}
\label{rem:nodalInterpolation}
In this manuscript we use the Cl{\'e}ment-type operator for the LOD construction, because its favorable connection to the $L^2$-projection. This will be exploited in the analysis of Section~\ref{sect:errAnalysis}. 
There are, however, many other choices of (quasi-) interpolation operators possible and we refer to~\cite{EngHMP19} for a review.
In particular, for $d=2$ and, thus, a one-dimensional boundary $\Gamma$, the nodal interpolation operator is a valid choice. It immediately leads to completely localized corrector problems, i.e., $\calG_T$ in \eqref{eq:corrector-element} is automatically zero outside the element $T$ so that the localization step \eqref{eq:corrector-element-local} is not necessary with this choice of the interpolation operator. 
\end{remark}
\begin{remark}
If we compute $\calG$ using $\fraka$ and the nodal interpolation operator for a one-dimensional boundary, we can explicitly characterize $\calG q_H$ for $q_H\in Q_H$.
This can be used to show that $\fraka((\id-\calG)p_H, q_H)=\int_\Omega (a_\eps)_{\mathrm{harm}} \nabla p_H\cdot \nabla q_H\dx$, where $(a_\eps)_{\mathrm{harm}}$ is an element-wise constant coefficient consisting of the harmonic average of $a_\eps$ on the element, see \cite{HenMPSVK20}.
Hence, the LOD stiffness matrix can be computed as a standard finite element stiffness matrix with a modified coefficient.
This alternative characterization will be applied in the numerical experiments in Section~\ref{sect:numerics:exp1}.
Moreover, in case of a periodic $a_\eps$, $(a_\eps)_{\mathrm{harm}}$ is exactly the effective coefficient from homogenization theory \cite{GalP17,PetVV19}.
\end{remark}
%
%
\section{Multiscale Error Analysis}\label{sect:errAnalysis}
For the analysis of the discretization error, we consider the PDAE in terms of the bilinear forms~$\frakk$, $\fraka$, and~$\frakb$ as introduced in~\eqref{eq:PDAE:bilinear}. 
In the following error analysis we focus on the multiscale phenomena on the boundary and fix the discrete spaces~$\Vh$ and~$\Mh$. Further, we only discuss the error which occurs due to the spatial discretization, i.e., we compare the exact with the semi-discrete solution. We consider the inf-sup stable pairing introduced in Proposition~\ref{prop:P1P1}, i.e.,  
\[
  \Vh\coloneqq\calP_1(\calTO)\subseteq \V, \qquad
  \Mh\coloneqq\calP_1(\calTO|_\Gamma)\subseteq \M 
\]
with corresponding mesh size~$H$. In the following, we discuss various choices for~$\Qh$ and start with the trivial case, in which~$\calTG$ coincides with~$\calTO|_\Gamma$ and~$\Qh=\Mh$. 

Recall that we write~$(\,\cdot\,, \cdot\,)_\Omega$ and~$(\,\cdot\,, \cdot\,)_\Gamma$ for the respective $L^2$-inner products on~$\Omega$ and~$\Gamma$. Accordingly, we denote the corresponding~$L^2$-norms by~$\|\cdot\|_\Omega$ and~$\|\cdot\|_\Gamma$. 
%
%
\subsection{Special case~$\Qh = \Mh = \Vh|_\Gamma$}\label{sect:errAnalysis:KovL}
Assume $\Qh = \calP_1(\calTG)$ with $\calTG = \calTO|_\Gamma$, i.e., the discrete space for $p$ equals $\Mh$, which itself equals~$\Vh$ restricted to the boundary. 
We show that in this special case we regain the discretization proposed in~\cite{KovL17} and thus, may pick up the corresponding convergence results. 
The semi-discrete system reads
\begin{subequations}
\label{eq:semidiscreteKonL}
\begin{align}
	(\dot u_H, v_H)_\Omega + (\dot p_H, q_H)_\Gamma 
	+ \frakk(u_H,v_H) + \fraka(p_H,q_H)
	+ \frakb(q_H-v_H|_\Gamma, \lambda_H)
	&= (f, v_H)_\Omega + (g, q_H)_\Gamma, \label{eq:semidiscreteKonL:a} \\
	\frakb(p_H-u_H|_\Gamma, \mu_H) &= 0 \label{eq:semidiscreteKonL:b}
\end{align}
\end{subequations}
for all test functions~$v_H\in\Vh$, $q_H\in\Qh$, and $\mu_H\in\Mh$. 
A key property in this special case is that~$p_H=u_H|_\Gamma$ along the boundary, i.e., the original coupling condition is satisfied pointwise also for the semi-discrete solution. Recall that this is automatically satisfied in the formulation of~\cite{KovL17}, since there is only a single discrete variable.  
\begin{lemma}
\label{lem_pHequalsuH}
Given meshes $\calTG = \calTO|_\Gamma$ and discrete spaces $\Vh=\calP_1(\calTO)$, $\Qh=\Mh=\calP_1(\calTO|_\Gamma)$, the semi-discrete solution satisfies~$p_H=u_H|_\Gamma$ for all times. 
\end{lemma}
\begin{proof}
Due to the definition of the discrete spaces we have~$p_H-u_H|_\Gamma\in\Mh$.  
Thus, it depicts a valid test function in~\eqref{eq:semidiscreteKonL:b}, leading to 
\[
  \| p_H-u_H|_\Gamma\|^2_{\Gamma}
  = \frakb(p_H-u_H|_\Gamma,p_H-u_H|_\Gamma)
  =0. \qedhere
\]	
\end{proof}
\begin{remark}
The previous result remains true if $\calTO|_\Gamma$ is a refinement of $\calTG$, since this still implies $p_H=u_H|_\Gamma$ for the semi-discrete solution. 
\end{remark}
Lemma~\ref{lem_pHequalsuH} indicates that we may eliminate the variable~$p_H$ from the system, since it contains only redundant information. 
Further, we can eliminate the Lagrange multiplier by considering test functions of the form~$v_H\in\Vh$, $q_H=v_H|_\Gamma\in\Qh$, since this turns~\eqref{eq:semidiscreteKonL:a} into
\[
  (\dot u_H, v_H)_\Omega + (\dot u_H|_\Gamma, v_H|_\Gamma)_\Gamma 
  + \frakk(u_H,v_H) + \fraka(u_H|_\Gamma,v_H|_\Gamma)
  = (f, v_H)_\Omega + (g, v_H|_\Gamma)_\Gamma. 
\]
Note that this is nothing else than the Galerkin discretization given in~\eqref{eq:KonL:semidiscrete}. 
Thus, all error estimates derived in~\cite{KovL17} hold for the considered case. 
For $H^2$-regular solutions this leads to the following result. 
\begin{theorem}[{cf.~\cite[Th.~3.2]{KovL17}}]
\label{thm:KovL:polyhedral}
Consider a polyhedral domain~$\Omega$ with solution~$u\in H^1(0,T;H^2(\Omega))$ satisfying~$p=u|_\Gamma \in H^1(0,T;H^2(\calTO|_\Gamma))$, i.e., $u|_\Gamma$ is piecewise $H^2$ on the boundary.   
Then, there exists a constant $C(u)>0$ such that 
\begin{align*}
  \| u(t) - u_H(t) \|_{\Omega} + \| p(t) - p_H(t) \|_{\Gamma}  
  \le C(u)\, H.
\end{align*}
Note, however, that the constant~$C(u)$ depends on the bilinear forms~$\fraka$ and~$\frakk$. In general, this includes a dependence on~$\eps$ with a negative power.
\end{theorem}
Finally, we would like to mention the convergence result on smooth domains, namely 
\[ 
  \| u(t) - u_H(t) \|_{\Omega} + \| p(t) - p_H(t) \|_{\Gamma} 
  \le C(u)\, H^2
\]
for $u\in H^1(0,T;H^2(\Omega))$ with~$u|_\Gamma \in H^1(0,T;H^2(\Gamma))$, see~\cite[Th.~3.6]{KovL17}. 
Again, this result is based on a Ritz projection, which involves a dependence on the inverse of~$\eps$. This then leads to a pre-asymptotic effect for coarse mesh sizes, cf.~the numerical experiments in Section~\ref{sect:numerics}.
 
Recall that we do not consider smooth domains in this paper but rather Lipschitz domains. Further, we do not assume the solution to be~$H^2$-regular and aim to find approximation results, which do not involve $\eps$-dependencies. 
%
\subsection{LOD on the boundary}\label{sect:errAnalysis:LOD}
We now turn to the case of interest, in which the discretization on the boundary is obtained by the LOD as described in Section \ref{sect:disc:LOD}. 
We introduce the space $\Qh = \calP_1(\calT_{\Gamma, H})$ and the LOD space
\[
\tQh
\coloneqq(\operatorname{id}-\calG_m)\,\Qh.
\]
Note that $\tQh$ implicitly depends on the so-called oversampling parameter~$m$.
Furthermore, we have the relation $\Qh= I_H \tQh=\Pi_H \tQh$. More precisely, for any $\tilde{q}_H\in \tQh$, there exists a unique $q_H\in \Qh$ such that $\tilde{q}_H=(\id-\calG_m)q_H$.
Closely inspecting the definition of $\calG_m$, $I_H$, and $\Pi_H$ indeed reveals that $q_H=\Pi_H\tilde{q}_H$.
As already mentioned in Section \ref{sect:disc:LOD}, the coarse mesh on the boundary is given as the restriction of the bulk mesh, i.e., $\calT_{\Gamma, H}=\calTO|_\Gamma$. 

In the Petrov-Galerkin LOD (PG-LOD) approach, we use the ansatz spaces $\Vh$, $\tQh$, and $\Mh$ as above, but the test spaces $\Vh$, $\Qh$, and $\Mh$, i.e., the test functions are not modified in comparison to a classical approach.
This leads to the following variational formulation: Find $u_H\colon [0,T] \to \Vh$, $\tilde{p}_H\colon [0,T] \to \tQh$, and $\lambda_H\colon [0,T] \to \Mh$ such that
\begin{subequations}
\label{eq:semidiscretePGLOD}
\begin{align}
	(\dot u_H, v_H)_\Omega + (\dot{\tilde{p}}_H, q_H)_\Gamma 
	+ \frakk(u_H, v_H) + \fraka(\tilde{p}_H, q_H)
	+ \frakb(q_H- v_H|_\Gamma, \lambda_H)
	&= (f, v_H)_\Omega + (g, q_H)_\Gamma, \label{eq:semidiscretePGLOD:a} \\
	\frakb(\Pi_H\tilde{p}_H-u_H|_\Gamma, \mu_H) &= 0 \label{eq:semidiscretePGLOD:b}
\end{align}
\end{subequations}
for all test functions $v_H\in\Vh$, $q_H\in\Qh$, and $\mu_H\in\Mh$. 
Note that the PG-LOD approach~\eqref{eq:semidiscretePGLOD} is well-posed because of $\dim\tQh=\dim \Qh$.
Similar as in Lemma~\ref{lem_pHequalsuH}, we deduce $\Pi_H\tilde{p}_H=u_H|_\Gamma$ for all times, which allows us to eliminate $\tilde{p}_H$ from the system and moreover to remove the coupling term by considering only test functions $q_H=v_H|_\Gamma$. 
This then leads to the problem of finding~$u_H\colon [0,T] \to \Vh$ and $p_H\colon [0,T] \to \Qh$ such that 
\[
(\dot u_H, v_H)_\Omega + ( (\id-\calG_m) \dot{p}_H, v_H|_\Gamma)_\Gamma 
+ \frakk(u_H,v_H) + \fraka( (\id-\calG_m) p_H, v_H|_\Gamma)
= (f, v_H)_\Omega + (g, v_H|_\Gamma)_\Gamma.  
\]
Note that this is a LOD-variation of the Galerkin discretization~\eqref{eq:KonL:semidiscrete}. In the following, however, we proceed with the analysis of the full Petrov-Galerkin system~\eqref{eq:semidiscretePGLOD}. 
%
%
\subsection{Coupled Ritz projection}
Recall that~$\frakaTilde\colon \Q\times\Q\to\R$ is defined by
\[
  \frakaTilde(p, q)
  \coloneqq \fraka(p ,q) + (\alpha\,p ,q)_\Gamma,  
\]
where $\alpha>0$ is the lower bound on the diffusion coefficient $a_\eps$ such that~$\frakaTilde$ is elliptic on~$\Q$ with constant~$\alpha$.  
With this, we define a Ritz projection of Petrov-Galerkin type for given~$u\in\V$ and~$p\in\Q$. 
More precisely, we seek for $\calRu \coloneqq \calR_H^u(u,p) \in V_H$, $\calRp \coloneqq \calR_H^p(u,p) \in \tQh$, and $\calRl \coloneqq \calR_H^\lambda(u,p) \in \Mh$ such that 
\begin{subequations}
\label{eq:RitzPGLOD}
\begin{align}
	\frakk(\calRu,v_H) + \frakaTilde(\calRp,q_H)
	+ \frakb(q_H-v_H|_\Gamma, \calRl)
	&= \frakk(u, v_H) + \frakaTilde(p, q_H), \label{eq:RitzPGLOD:a} \\
	\frakb(\Pi_H \calRp - \calRu|_\Gamma, \mu_H) &= 0 \label{eq:RitzPGLOD:b}
\end{align}
\end{subequations}
for all test functions $v_H\in\Vh$, $q_H\in \Qh$, and $\mu_H\in\Mh$. 
Before discussing the approximation property of this projection, we need to guarantee the unique solvability of~\eqref{eq:RitzPGLOD}. 
\begin{lemma}
\label{lem:existenceRitz}
Given~$u\in\V$, $p\in\Q$, and~$m$ sufficiently large in the sense that~$\gamma^m\lesssim \alpha/C_{\tilde\fraka}$, system~\eqref{eq:RitzPGLOD} is well-posed, i.e., there exist unique $\calRu\in V_H$, $\calRp\in \tQh$, and $\calRl\in \Mh$. 
\end{lemma}
\begin{proof}
As a first step we rewrite~\eqref{eq:RitzPGLOD} as a standard saddle point problem with identical trial and test space. 
For this, we introduce the bilinear forms 
\[
  A_H((u,p),(v,q)) \coloneqq \frakk(u, v) + \frakaTilde(p, \Pi_H q),\qquad 
  B_H((u,p), \mu) \coloneqq \frakb(\Pi_Hp - u|_\Gamma, \mu). 
\]
The Ritz projection can now be equivalently characterized by 
\begin{align*}
  A_H((\calRu, \calRp),(v_H, \tilde q_H)) + B_H((v_H, \tilde q_H), \calRl)
  &= A_H((u, p), (v_H, \tilde q_H)), \\
  B_H((\calRu, \calRp), \mu_H) 
  &= 0 
\end{align*}
for all $v_H\in\Vh$, $\tilde q_H\in \tQh$, and $\mu_H\in\Mh$. 
Note that this employs the one-to-one relationship of the spaces~$\Qh$ and~$\tQh$. 

The inf-sup-condition of~$B_H$ follows directly from Proposition~\ref{prop:P1P1}.
It remains to show the coercivity of~$A_H$ on~$\ker B_H$. 
For this, we consider~$(v_H, \tilde q_H) \in \Vh\times\tQh$ with $\Pi_H \tilde q_H = v_H|_\Gamma$. 
Employing $\tilde q_H=(\id -\calG_m)\Pi_H \tilde q_H$, we deduce
\begin{align*}
A_H((v_H&,\tilde q_H), (v_H,\tilde q_H)) \\
&= \frakk(v_H, v_H) + \frakaTilde(\tilde q_H, \Pi_H \tilde q_H)\\
&= \frakk(v_H, v_H) + \frakaTilde((\id-\calG_m)\Pi_H\tilde q_H, \Pi_H \tilde q_H)\\
&= \frakk(v_H, v_H) + \frakaTilde((\id-\calG)\Pi_H\tilde q_H, \Pi_H \tilde q_H)+\frakaTilde((\calG-\calG_m)\Pi_H\tilde q_H, \Pi_H \tilde q_H).
\end{align*}
The definition of~$\calG$ and $W_h$ implies  
\begin{align*}
\frakaTilde((\id-\calG)\Pi_H\tilde q_H, \Pi_H \tilde q_H)
= \frakaTilde((\id-\calG)\Pi_H\tilde q_H, (\id-\calG)\Pi_H \tilde q_H).
\end{align*}
Hence, we obtain due to the ellipticity of~$\frakaTilde$ and~\eqref{eq:corrector:decay} that
\begin{align*}
A_H((v_H&, \tilde q_H), (v_H, \tilde q_H)) \\
&= \frakk(v_H, v_H) + \frakaTilde((\id-\calG)\Pi_H\tilde q_H, (\id-\calG)\Pi_H \tilde q_H)+\frakaTilde((\calG-\calG_m)\Pi_H\tilde q_H, \Pi_H \tilde q_H)\\
&\geq \kappa\, \|\nabla v_H\|^2_{\Omega} +\alpha\,\|(\id-\calG)\Pi_H\tilde q_H\|^2_\calQ-C_{\tilde\fraka}\,m^{(d-1)/2}\gamma^m\|\Pi_H \tilde q_H\|^2_\calQ.
\end{align*}
Note that we have the following norm equivalences
\[ 
\|\Pi_H\tilde q_H\|_\calQ= \|\Pi_H(\id-\calG)\Pi_H\tilde q_H\|_\calQ\lesssim \|(\id-\calG)\Pi_H\tilde q_H\|_\calQ
\]
and 
\[
\|\tilde q_H\|_\calQ 
= \|(\id-\calG_m)\Pi_H\tilde q_H\|_\calQ 
\lesssim \|\Pi_H \tilde q_H\|_\calQ. 
\]
With these estimates, $\Pi_H\tilde q_H=v_H|_\Gamma$, and~$\gamma^m\lesssim \alpha/C_{\tilde\fraka}$ we deduce that 
\begin{align*}
A_H((v_H, \tilde q_H), (v_H, \tilde q_H)) &\geq \kappa\, \|\nabla v_H\|^2_{\Omega} +\alpha\,\|(\id-\calG)\Pi_H\tilde q_H\|^2_\calQ-C_{\tilde\fraka}\,m^{(d-1)/2}\gamma^m\|\Pi_H \tilde q_H\|^2_\calQ\\
&\gtrsim \|(v_H, \Pi_H \tilde q_H)\|^2_{\calV\times \calQ} \\
&\gtrsim \|(v_H, \tilde q_H)\|^2_{\calV\times \calQ}.\qedhere
\end{align*}
\end{proof}
We now need to analyze the approximation properties of the Ritz projection. 
\begin{proposition}
\label{prop:errorRitz}
Given~$u\in \calV$ and $p\in \calQ$ with $u|_\Gamma = p$, the coupled Ritz projection defined in~\eqref{eq:RitzPGLOD} satisfies for sufficiently large~$m$ (i.e., $\gamma^m\lesssim \alpha/C_{\tilde\fraka}$) the estimate 
\[
  \|u-\calRu\|_\calV + \|p-\calRp \|_\calQ
  \ \lesssim\ \inf_{v_H\in\Vh} \|u-v_H\|_\calV + \inf_{\tilde q_H\in \tQh} \|p - \tilde q_H\|_\calQ.
\]
\end{proposition}
\begin{proof}
The idea of the proof is to use the reformulation of~\eqref{eq:RitzPGLOD} as in the proof of Lemma~\ref{lem:existenceRitz} and to apply the techniques presented in~\cite[Ch.~II.2]{BreF91}. 
By the definition of the Ritz projection we have for $(v_H,\tilde q_H) \in \ker B_H$ that  
\[
  A_H((u-\calRu, p-\calRp),(v_H, \tilde q_H)) 
  = B_H((v_H, \tilde q_H), \calRl)
  = 0. 
\]
With the coercivity of $A_H$, which was shown in the proof of Lemma~\ref{lem:existenceRitz}, we obtain for arbitrary $w_H\in\Vh$, $\tilde r_H\in\tQh$ with $w_H|_\Gamma=\Pi_H\tilde r_H$ the estimate 
\begin{align*}
  \|(w_H - \calRu, \tilde r_H - \calRp)\|_{\V\times\Q}
  &\ \lesssim\ \sup_{(v_H, \tilde q_H) \in \ker B_H} \frac{ A_H((w_H - \calRu, \tilde r_H - \calRp), (v_H, \tilde q_H))}{\|(v_H, \tilde q_H)\|_{\V\times\Q}} \\
  &\ =\ \sup_{(v_H, \tilde q_H) \in \ker B_H} \frac{ A_H((w_H - u, \tilde r_H - p), (v_H, \tilde q_H))}{\|(v_H, \tilde q_H)\|_{\V\times\Q}} \\
  &\ \lesssim\ \|(u - w_H, p - \tilde r_H)\|_{\V\times\Q}.
\end{align*}
Thus, by triangle inequality it holds that  
\[
  \|(u - \calRu, p - \calRp)\|_{\V\times\Q}
  \ \lesssim\ \inf_{(w_H, \tilde r_H) \in \ker B_H} \|(u - w_H, p - \tilde r_H)\|_{\V\times\Q}.
\]
Following~\cite[Prop.~II.2.5]{BreF91}, we conclude with the inf-sup property of~$B_H$ that 
\[
  \inf_{(w_H, \tilde r_H) \in \ker B_H} \|(u - w_H, p - \tilde r_H)\|_{\V\times\Q}
  \ \lesssim\ \inf_{w_H \in\Vh, \tilde r_H \in \tQh} \|(u - w_H, p - \tilde r_H)\|_{\V\times\Q},
\]
which provides the stated decoupled estimate.  
\end{proof}
For the two previous results we had to assume that the localization parameter~$m$ is sufficiently large compared to the contrast of the diffusion coefficient~$a_\eps$. 
To ensure the full convergence order later on, we will also need to assume that~$m$ is sufficiently large compared to the (coarse) mesh size~$H_\Gamma$. We summarize this in the following assumption. 
\begin{assumption}[localization parameter]
\label{ass:m}
We assume that $m$ is sufficiently large in the sense that~$\gamma^m\lesssim \alpha/C_{\tilde\fraka}$ and~$m\gtrsim |\log H|$. 
\end{assumption}
In the following, we show a priori estimates for the coupled Ritz projections based on the quasi-optimality of Proposition~\ref{prop:errorRitz}. 
For a precise formulation, we need some further notation.
Recall that the LOD on the boundary in Section~\ref{sect:disc:LOD} utilized a fine-scale mesh~$\calT_{\Gamma, h}$ with associated finite element space $\calP_1(\calT_{\Gamma, h})$ for the definition of the correctors. 
Further assume that~$\calT_{\Gamma, h}$ is the restriction of a volume mesh $\calT_{\Omega, h}$ and denote by $(u_h, p_h)$ the finite element solution corresponding to~\eqref{eq:PDAE:bilinear} on $\calP_1(\calT_{\Omega, h})\times \calP_1(\calT_{\Gamma, h})$.
This solution is never computed in practice and only serves as a reference solution.
We assume~$h$ to be sufficiently small so that~$p_h$ is a good approximation of $p$, i.e., the error~$\|p-p_h\|_\calQ$ is sufficiently small. Note that the error of the fine-scale discretization can be estimated with the help of Theorem~\ref{thm:KovL:polyhedral}.
Furthermore, we introduce the discrete operator $\widetilde{\calA}_{\eps, h}\colon  \calP_1(\calT_{\Gamma, h})\to \calP_1(\calT_{\Gamma, h})$ via
\begin{equation}\label{eq:Aepsh}
(\widetilde{\calA}_{\eps, h} p_h, q_h)_\Gamma = \frakaTilde(p_h, q_h)
\qquad \text{for all}\quad p_h, q_h\in  \calP_1(\calT_{\Gamma, h}).
\end{equation}
Thus, the operator~$\widetilde{\calA}_{\eps, h}$ is the $L^2$-representative of~$\frakaTilde$, restricted to the fine-scale finite element space. 
\begin{corollary}\label{cor:ritz}
Consider~$u\in H^s(\Omega) \subseteq \calV$, $1\le s\le 2$, and $p\in\calQ$ with $u|_\Gamma=p$. Further, let~$m$ satisfy Assumption~\ref{ass:m}. 
Then, we have that
\[
  \|u-\calRu\|_\calV + \|p-\calRp \|_\calQ 
  \ \lesssim\ H^{s-1}|u|_{H^s(\Omega)}  + H\,\|\widetilde{\calA}_{\eps, h}\, p_h\|_{\Gamma} + \|p-p_h\|_\calQ.
\]
\end{corollary}
\begin{proof}
Due to Proposition~\ref{prop:errorRitz}, we only need to estimate the best approximation errors $\inf_{v_H\in\Vh} \|u-v_H\|_\calV$  and $\inf_{q_H\in \tQh} \|p-q_H\|_\calQ$.
The error for $u$ follows by standard interpolation estimates.
For the error in $p$, we use the triangle inequality and obtain
\[
  \inf_{\tilde q_H\in \tQh} \|p-\tilde q_H\|_\calQ
  \leq \|p-p_h\|_\calQ + \inf_{\tilde q_H\in \tQh} \|p_h- \tilde q_H\|_\calQ.
\]
Choosing $\tilde q_H=(\id-\calG_m)\Pi_H p_h$, the last term is estimated in a standard LOD manner, see, e.g., \cite{MalP18,ElfGH15}. 
\end{proof}

Neglecting the fine-scale discretization error (i.e., choosing~$h\ll \eps$), Corollary~\ref{cor:ritz} can be summarized as
\begin{equation}\label{eq:ritz:energy}
\|u-\calRu\|_\calV + \|p-\calRp \|_\calQ 
  \ \lesssim\ H^{s-1} + H. 
\end{equation}
%
%
\subsection{$L^2$-estimates}
To show corresponding $L^2$-estimates for the Ritz projections, we consider the following auxiliary problem: Seek~$w\in \calV$, $r\in \calQ$, and~$\lambda^z\in \calM$ such that
\begin{subequations}
\label{eq:ritz:dual}
\begin{align}
	\frakk(w,v) + \frakaTilde(r,q)
	+ \frakb(q-v|_\Gamma, \lambda^z)
	&=(u-\calRu, v)_\Omega + (p-\Pi_H\calRp, q)_\Gamma,\\
	\frakb(r - w|_\Gamma, \mu) &= 0
\end{align}
\end{subequations}
for all $v\in\calV$, $q\in\calQ$, and $\mu\in \calM$.
Note that this is similar to the stationary part of~\eqref{eq:PDAE:bilinear} with $\fraka$ replaced by $\frakaTilde$ and adjusted source terms on the right-hand side.
Hence, we can expect the same spatial regularity for~$(w,r)$ to hold, since $u-\calRu\in L^2(\Omega)$ and $p-\Pi_H\calRp\in L^2(\Gamma)$.
Furthermore, we introduce 
\[
\wcba(r, \Qh):=\inf_{r_H\in \Qh}\frac{\|r-r_H\|_\calQ}{\|u-\calRu\|_\Omega+\|\Pi_H(p-\calRp)\|_\Gamma}
\]
and
\[
 \wcba(p, \Qh):=\inf_{q_H\in \Qh} \|p-q_H\|_\calQ.
\]
These {\em worst case best-approximation} errors of $\Qh$ for $r$ and $p$ with respect to the energy norm are bounded independently of $\eps$ without any further regularity assumptions and are of order $H^{\sigma-1}$ for sufficiently regular~$r$ and $p$ in~$H^\sigma(\Gamma)$ for $1\leq \sigma\leq2$.
We emphasize, however, that exploiting higher regularity to estimate $\wcba$ may introduce an $\eps$-dependency.
For readability, we will omit the fine-scale discretization error in the following proposition.
\begin{proposition}
\label{prop:ritz:L2}
Let~$u\in H^s(\Omega)\subseteq \calV$, $1\le s\le 2$, and $p\in \calQ$ be given with $u|_\Gamma=p$. Further assume that the unique solution $(w,r)\in \calV\times\calQ$ to~\eqref{eq:ritz:dual} satisfies $w\in H^s(\Omega)$. 
Then we have with Assumption~\ref{ass:m} that
\begin{equation}\label{eq:ritz:L2-1}
\begin{aligned}
\|u-\calRu\|_\Omega + \|p-\Pi_H\calRp\|_\Gamma
&\lesssim (H^{s-1}+\wcba(r,\Qh) + H)\, \|(u-\calRu, p-\calRp)\|_{\calV\times\calQ}\\
&\qquad + H\, \wcba(p,\Qh) + H^2\, (\|p\|_\calQ + \|\widetilde{\calA}_{\eps, h}p_h\|_\Gamma).
\end{aligned}
\end{equation}
\end{proposition}
Before proving Proposition~\ref{prop:ritz:L2}, we discuss the obtainable rates for the Ritz projection in $L^2$.
\begin{remark}\label{rem:ritzL2}
As mentioned before, the worst case best-approximation errors are of order $H^{\sigma-1}$ if $r$ and $p$ are in $H^{\sigma}(\Gamma)$, respectively.
Combining Proposition~\ref{prop:ritz:L2} and Corollary~\ref{cor:ritz}, we can thus summarize that
\begin{equation}\label{eq:ritz:L2-short}
\|u-\calRu\|_\Omega + \|p-\Pi_H\calRp\|_\Gamma\lesssim H^{2(s-1)}+H^{s+\sigma-2}+H^s+H^\sigma+H^2.
\end{equation}
For optimal regularity $\sigma=s=2$, we obtain in agreement with~\cite{KovL17} the expected quadratic rate for the Ritz error. Note that in order to have this optimal regularity, $\Omega$ needs to be convex.
In the worst case $\sigma=1$ (see Section~\ref{sect:model:regularity}), estimate~\eqref{eq:ritz:L2-short} results in $H^{s-1}$ as dominant term, which is comparable to the energy norm estimate in Corollary~\ref{cor:ritz}.
Note that this rate seems rather pessimistic, but is explained by the low regularity of $p$ for general coefficients $a_\eps$. 
Finally we observe that if $s=\sigma+1/2$, we obtain the rate $H^{2(s-5/4)}$ from~\eqref{eq:ritz:L2-short}, which is better than $H^{s-1}$ only if $s\geq 3/2$. 
\end{remark}
\begin{proof}[Proof of Proposition~\ref{prop:ritz:L2}]
Inserting $v=u-\calRu$ and $q=p-\Pi_H\calRp$ into~\eqref{eq:ritz:dual} and observing that $q=v|_\Gamma$, we obtain
\begin{align*}
&\|u-\calRu\|_\Omega^2+\|p-\Pi_H\calRp\|_\Gamma^2\\
&\qquad=\frakk(u-\calRu, w)+\frakaTilde(p-\Pi_H\calRp, r)\\
&\qquad=\frakk(u-\calRu, w)+\frakaTilde(p-\calRp, r)+\frakaTilde(\calRp-\Pi_H\calRp, r)\\
&\qquad=\frakk(u-\calRu, w-w_H)+\frakaTilde(p-\calRp, r-r_H)+\frakaTilde(\calRp-\Pi_H\calRp, r)
\end{align*}
for any $w_H\in \Vh$ and $r_H\in \Qh$ with $w_H|_\Gamma=r_H$.
In the last step, we have used the definition of~$\calG$ and the Galerkin-type orthogonality
\[
  \frakk(u-\calRu, w_H)+\frakaTilde(p-\calRp, r_H) = 0, 
\]
which follows from the definition of the coupled Ritz projection in~\eqref{eq:RitzPGLOD}. 
Using once more the inf-sup stability of~$\frakb$ and~\cite[Prop.~II.2.5]{BreF91}, we obtain 
\begin{align*}
&\frakk(u-\calRu, w-w_H)+\frakaTilde(p-\calRp, r-r_H)\\
&\qquad\qquad\lesssim \|(u-\calRu, p-\calRp)\|_{\calV\times \calQ}\ \Big(
\inf_{w_H\in\Vh}\|w-w_H\|_\calV + \inf_{r_H\in \Qh}\|r-r_H\|_\calQ \Big). 
\end{align*}
Further, the application of standard interpolation estimates yields 
\begin{align*}
\inf_{w_H\in\Vh}\|w-w_H\|_\calV + \inf_{r_H\in \Qh}\|r-r_H\|_\calQ 
&\lesssim H^{s-1}|w|_{H^s(\Omega)}+\inf_{r_H\in \Qh}\|r-r_H\|_\calQ\\
&\lesssim (H^{s-1}+\wcba(r,\Qh))\, \big( \|u-\calRu\|_\Omega + \|p-\calRp\|_\Gamma \big).
\end{align*}
Recalling $\calRp=(\id-\calG_m)\Pi_H\calRp$, we obtain
\begin{align*}
&\frakaTilde(\calRp-\Pi_H\calRp, r)\\
&\qquad=-\frakaTilde(\calG_m\Pi_H\calRp, r)\\
&\qquad=\frakaTilde(\calG_m\Pi_H(p-\calRp), r)+\frakaTilde((\calG-\calG_m)\Pi_H p, r)-\frakaTilde(\calG\Pi_H p, r)\\
&\qquad=\frakaTilde(\calG_m\Pi_H(p-\calRp), r-(\id-\calG)\Pi_H r)+\frakaTilde((\calG-\calG_m)\Pi_H p, r-(\id-\calG)\Pi_H r)\\
&\qquad\qquad-\frakaTilde(\calG\Pi_H p, r-(\id-\calG)\Pi_H r)\\
&\qquad\lesssim \bigl(\|p-\calRp\|_\calQ + m^{(d-1)/2}\gamma^m \|p\|_\calQ + \|\calG\Pi_H p\|_\calQ\bigr)\, \|r-(\id-\calG)\Pi_H r\|_\calQ,
\end{align*}
where we used the stability of $\calG_m$ and $\Pi_H$ in the last step.
With standard LOD estimates \cite{MalP18,ElfGH15} we obtain
\begin{align*}
\|r-(\id-\calG)\Pi_H r\|_\calQ
&\lesssim \|r-r_h\|_\calQ + H\, \|\tilde \calA_{\eps, h}r_h\|_\Gamma\\
&\lesssim \|r-r_h\|_\calQ + H\, (\|u-\calRu\|_\Omega + \|p-\calRp\|_\Gamma).
\end{align*}
It remains to bound $\calG\Pi_H p$ from above. We have that
\[
-\calG \Pi_H p= (\id-\Pi_H)p + (\id-\Pi_H)((\id-\calG)\Pi_Hp -p),
\]
where the first term can be bounded by $\wcba(p, \Qh)$. For the second term it follows again by standard LOD estimates that
\[
\|(\id-\Pi_H)((\id-\calG)\Pi_Hp -p)\|_\calQ \lesssim \|p-p_h\|_\calQ + H\,\|\tilde \calA_{\eps,h}p_h\|_\Gamma.
\]
Combining all foregoing estimates finishes the proof.
\end{proof}
Note that a similar result to Proposition~\ref{prop:ritz:L2} can also be established for $\|p-\calRp\|_\Gamma$ by writing $p-\calRp=p-\Pi_H\calRp+\Pi_H\calRp-\calRp$ and using the properties of $\Pi_H$.
A more careful analysis of $\|p-\calRp\|_\Gamma$ is omitted to keep the manuscript at a considerable length.
\begin{remark}
If we write the Ritz problem~\eqref{eq:RitzPGLOD} in a Galerkin form, i.e., 
\begin{align*}
	\frakk(\calRu,v_H) + \frakaTilde(\calRp,\tilde q_H)
	+ \frakb(\Pi_H\tilde q_H-v_H|_\Gamma, \calRl)
	&= \frakk(u, v_H) + \frakaTilde(p, \tilde q_H), \\
	\frakb(\Pi_H \calRp - \calRu|_\Gamma, \mu_H) &= 0
\end{align*}
with test functions $v_H\in\Vh$, $\tilde q_H\in \tQh$, $\mu_H\in\Mh$, then the right-hand side~$p-\Pi_H\calRp$ in the auxiliary problem~\eqref{eq:ritz:dual} is modified to~$p-\calRp$.
We then deduce
\begin{align*}
&\|u-\calRu\|^2_\Omega + \|p-\calRp\|_\Gamma^2\\
&\qquad\lesssim \|(u-\calRu, p-\calRp)\|_{\calV\times\calQ}\inf_{w_H|_\Gamma=\Pi_H\tilde r_H}\|(w-w_H, r-\tilde r_H)\|_{\calV\times \calQ} + |\frakb(\Pi_H\calRp-\calRp, \lambda^z)|.
\end{align*}
At this point, the key is the regularity and characterization of $\lambda^z$.
A priori, we only have~$\lambda^z\in H^{-1/2}(\Gamma)$, but for sufficiently regular $w$ one can deduce that~$\lambda^z\in L^2(\Gamma)$.
This would allow to estimate $\Pi_H\calRp-\calRp$ in $L^2(\Gamma)$, which leads to an order $H$ for this term and, on the whole, to 
\[
\|u-\calRu\|_\Omega + \|p-\calRp\|_\Gamma\lesssim H^{2(s-1)}+H.
\]
This seems to be better than the rate $H^{s-1}$ obtained in the worst case of Proposition~\ref{prop:ritz:L2}.
We emphasize, however, that this reasoning might only be valid if we have $s=2$, when we might expect $\sigma>1$ as well. Recall that in order to have $s=2$, the domain $\Omega$ needs to be convex.
\end{remark}
With the Ritz projection in hand, we are now able to estimate the error caused the presented multiscale finite element scheme. 
%
\subsection{Error estimate}
Let~$(u,p,\lambda)$ denote the exact solution of~\eqref{eq:PDAE:bilinear} and~$(u_H,\tilde p_H,\lambda_H)$ the PG-LOD solution defined through~\eqref{eq:semidiscretePGLOD}. 
As usual for the numerical analysis of parabolic systems, we decompose the errors in~$u$ and~$p$ with the help of the previously defined Ritz projection, which is applied pointwise in time. 
Thus, we consider  
\begin{align*}
  u - u_H 
  &= (u - \calRu) + (\calRu- u_H) =: \rho_u + \theta_u, \\
  p - \tilde{p}_H 
  &= (p - \calRp) + (\calRp - \tilde{p}_H) =: \rho_p + \theta_p.
\end{align*}
The main result of this paper is the following convergence result.
\begin{theorem}
\label{thm:main}
Consider a Lipschitz domain~$\Omega$ and the exact solution of~\eqref{eq:PDAE:bilinear} satisfying~$u\in H^1(0,T;H^s(\Omega))$ and~$p\in H^1(0,T;H^1(\Gamma))$, $1\le s\le 2$. 
Further assume that~$m$ satisfies Assumption~\ref{ass:m}. 
Let the initial values be chosen such that~$\|u_H(0)-\calRu(0)\|_{\Omega}=\calO(H)$ and $\|\tilde{p}_H(0)-\calRp(0)\|_{\Gamma}=\calO(H)$.
Then, we have the a priori error estimate
\begin{align*}
\|u(t)-u_H(t)\|_\Omega 
+ \|p(t)-\Pi_H \tilde p_H(t)\|_\Gamma 
\lesssim C(u,p)\, H^{s-1}.
\end{align*}
\end{theorem}
Note that we have only assumed the lowest possible regularity on~$p$.
We emphasize that in view of Remark~\ref{rem:ritzL2} higher convergence rates like $H^{2(s-1)}+H^{s+\sigma-2}$ can be obtained for $p\in H^1(0,T;H^\sigma(\Gamma))$ with $1\leq\sigma\leq2$. 
\begin{proof}
The Ritz errors $\rho_u$ and $\rho_p$ were already estimated in Corollary~\ref{cor:ritz} and Proposition~\ref{prop:ritz:L2}. Applying these estimates also for the first time derivatives, we get 
\begin{align*}
  \| p-\Pi_H\calRp \|_\Gamma+\|\rho_p\|_{\Gamma}
  &\lesssim H^{s-1}|u|_{H^s(\Omega)} + H\, \|\widetilde \calA_{\eps,h} p_h\|_\Gamma,\\ 
  \|\dot \rho_p\|_{\Gamma} + \|\dot\rho_u\|_{\Omega}
  &\lesssim H^{s-1}|\dot u|_{H^s(\Omega)} + H\, \|\widetilde \calA_{\eps,h} \dot p_h\|_\Gamma.
\end{align*}
Thus, it remains to estimate~$\theta_u$ and~$\theta_p$. 
Using~\eqref{eq:PDAE:bilinear}, \eqref{eq:semidiscretePGLOD}, and the definition of the coupled Ritz projection~\eqref{eq:RitzPGLOD}, we note that the pair~$(\theta_u,\theta_p)$ satisfies for all test functions $v_H\in \Vh$ and $q_H\in\Qh$, 
\begin{align*}
&(\dot \theta_u, v_H)_\Omega + (\dot\theta_p, q_H)_\Gamma + \frakk(\theta_u, v_H) + \fraka(\theta_p, q_H)\\
&\quad=(\calRdu, v_H)_\Omega + (\calRdp, q_H)_\Gamma + \frakk(\calRu, v_H) + \fraka(\calRp, q_H)
+ \frakb(q_H-v_H|_\Gamma, \lambda_H)\\
&\qquad\quad - (f, v_H)_\Omega - (g, q_H)_\Gamma \\   
&\quad=(\dot u - \dot \rho_u, v_H)_\Omega + (\dot p - \dot \rho_p, q_H)_\Gamma + \frakk(u, v_H) - \frakb(v_H|_\Gamma-q_H,\calRl) + \fraka(p, q_H)\\
&\qquad\quad + (\alpha \rho_p, q_H)_\Gamma
+ \frakb(q_H-v_H|_\Gamma, \lambda_H) - (f, v_H)_\Omega - (g, q_H)_\Gamma \\   
&\quad= - (\dot \rho_u, v_H)_\Omega - (\dot \rho_p, q_H)_\Gamma - \frakb(v_H|_\Gamma-q_H,\calRl) + (\alpha \rho_p, q_H)_\Gamma
+ \frakb(q_H-v_H|_\Gamma, \lambda_H-\lambda). 
\end{align*}
Since~$\theta_p \in \tQh$ is not allowed as test function, the natural choice along with~$v_H=\theta_u$ is~$q_H=\Pi_H\theta_p\in\Qh$. 
Recall the $L^2(\Gamma)$-orthogonality of $\Qh$ and $W_h$ from~\eqref{eq:IH:L2ortho1} such that one deduces  
\[
\frac12 \ddt \| \Pi_H \theta_p\|^2_\Gamma
= (\Pi_H \dot \theta_p, \Pi_H \theta_p)_\Gamma
= (\dot \theta_p, \Pi_H\theta_p)_\Gamma.
\]
With this, we get 
\begin{align*}
&\frac12 \frac{\text{d}}{\dt}\|\theta_u\|^2_\Omega 
+ \frac12 \frac{\text{d}}{\dt}\|\Pi_H\theta_p\|^2_\Gamma 
+ \kappa\, \|\nabla \theta_u\|^2_\Omega 
+ \alpha\,\|\nabla_\Gamma \theta_p\|^2_\Gamma\\
&\qquad\le (\dot \theta_u, \theta_u)_\Omega + (\dot\theta_p, \Pi_H \theta_p)_\Gamma + \frakk(\theta_u, \theta_u) + \fraka(\theta_p, \theta_p)\\  
&\qquad= (\dot \theta_u, \theta_u)_\Omega + (\dot\theta_p, \Pi_H\theta_p)_\Gamma + \frakk(\theta_u, \theta_u) + \fraka(\theta_p, \Pi_H \theta_p)
+ \fraka(\theta_p, \theta_p - \Pi_H \theta_p). 
\end{align*}
By~\eqref{eq:RitzPGLOD:b} and~\eqref{eq:semidiscretePGLOD:b} we conclude that the trace of~$\theta_u$ equals~$\Pi_H\theta_p$. 
Thus, the terms~$\frakb(q_H-v_H|_\Gamma, \lambda_H-\lambda)$ and~$\frakb(v_H|_\Gamma-q_H,\calRl)$  vanish and we conclude that 
\begin{align*}
&\frac12 \frac{\text{d}}{\dt}\|\theta_u\|^2_\Omega 
+ \frac12 \frac{\text{d}}{\dt}\|\Pi_H\theta_p\|^2_\Gamma 
+ \kappa\, \|\nabla \theta_u\|^2_\Omega 
+ \alpha\,\|\nabla_\Gamma \theta_p\|^2_\Gamma\\
&\qquad\le - (\dot \rho_u, \theta_u)_\Omega - (\dot \rho_p, \Pi_H \theta_p)_\Gamma + (\alpha \rho_p, \Pi_H \theta_p)_\Gamma 
+ \fraka(\theta_p, \theta_p - \Pi_H \theta_p).
\end{align*}
For the last term on the right-hand side, we deduce due to the definition of $\calG$, the exponential decay~\eqref{eq:corrector:decay}, and the stability and approximation properties of $\Pi_H$ that 
\begin{align*}
\fraka(\theta_p, (\id-\Pi_H)\theta_p)
&= \frakaTilde (\theta_p, (\id-\Pi_H)\theta_p)-(\alpha \theta_p, (\id-\Pi_H)\theta_p)_\Gamma\\
&= \frakaTilde((\id-\calG_m)\Pi_H\theta_p, (\id-\Pi_H)\theta_p)-(\alpha \theta_p, (\id-\Pi_H)\theta_p)_\Gamma\\
&=\frakaTilde((\calG-\calG_m)\Pi_H\theta_p, (\id-\Pi_H)\theta_p)-(\alpha \theta_p, (\id-\Pi_H)\theta_p)_\Gamma\\
&\lesssim (m^{(d-1)/2}\gamma^m+\alpha H)\, \|\nabla_\Gamma \theta_p\|^2_\Gamma. 
\end{align*}
Hence, the term $\fraka(\theta_p, (\id-\Pi_H)\theta_p)$ can be absorbed on the left-hand side for sufficiently large~$m$ and $H\lesssim 1$. 
Thus, we obtain with $\|(v, q)\|^2 \coloneqq \|v\|^2_\Omega + \|q\|^2_\Gamma$ the estimate 
\begin{align*}
\|(\theta_u, \Pi_H\theta_p)\|\,\ddt&\, \|(\theta_u, \Pi_H\theta_p)\|\\
&\quad= \frac12\,  \ddt\, \|(\theta_u, \Pi_H\theta_p)\|^2\\
%
%
&\quad\lesssim - (\dot \rho_u, \theta_u)_\Omega 
-  (\dot \rho_p, \Pi_H \theta_p)_\Gamma 
+ (\alpha\, \rho_p, \Pi_H \theta_p)_\Gamma \\
&\quad\le \|\dot \rho_u\|_\Omega \|\theta_u\|_\Omega 
+ \|\dot \rho_p\|_\Gamma \|\Pi_H\theta_p\|_\Gamma
+ \alpha\, \| \rho_p\|_\Gamma \|\Pi_H\theta_p\|_\Gamma \\
&\quad\le \bigl( \|\dot \rho_u\|_\Omega + \|\dot \rho_p\|_\Gamma + \alpha\, \|\rho_p\|_\Gamma \bigr)\, \|(\theta_u , \Pi_H \theta_p)\|.
\end{align*}
Thus, division by $\|(\theta_u , \Pi_H \theta_p)\|$, integration over time,  and taking squares results in
\begin{align*}
  &\|\theta_u(t)\|^2_\Omega 
  + \|\Pi_H\theta_p(t)\|^2_\Gamma \\
  &\quad\lesssim \|\theta_u(0)\|^2_\Omega + \|\Pi_H\theta_p(0)\|^2_\Gamma  
  +  \int_0^t \|\dot \rho_u(\tau)\|^2_\Omega  + \|\dot \rho_p(\tau)\|^2_\Gamma 
  + \alpha^2\| \rho_p(\tau)\|^2_\Gamma \dtau \\
 &\quad \lesssim \|\theta_u(0)\|^2_\Omega + \|\Pi_H\theta_p(0)\|^2_\Gamma \\ 
 &\qquad+  \int_0^t H^{2(s-1)} \big( |u(\tau)|_{H^s(\Omega)}^2 + |\dot u(\tau)|_{H^s(\Omega)}^2\big) + H^2 \big(\|\widetilde \calA_{\eps, h}p_h(\tau)\|_\Gamma^2 + \|\widetilde \calA_{\eps, h}\dot p_h(\tau)\|_\Gamma^2\big) \dtau.
\end{align*}
Due to the assumption on the initial values we have~$\|\theta_u(0)\|_\Omega^2=\calO(H^2)=\|\Pi_H\theta_p(0)\|_\Gamma^2$. 
Further, since~$u-u_H=\rho_u+\theta_u$ and $p-\Pi_H\tilde p_H=p-\Pi_H\calRp+\Pi_H\theta_p$, the combination of the estimates for the Ritz projection and for $(\theta_u, \Pi_H\theta_p)$ finishes the proof.
\end{proof}
The numerical verification that the obtained convergence rates are indeed $\eps$-independent, is subject of the following section. 
%
%
\section{Numerical Examples}\label{sect:numerics}
This final section is devoted to the numerical verification of the obtained convergence results. In particular, we will investigate the following questions: 
\begin{itemize}
	\item convergence behavior for smooth and discontinuous coefficients~$a_\eps$,
	\item convergence in the presence of mixed boundary conditions,
	\item influence of the localization parameter $m$,
	\item applicability of the nodal interpolation operator, and
	\item benefits of mesh refinements on the boundary only.
\end{itemize}
All examples consider~$\Omega=(0,1)^2$, a time horizon~$T=0.1$, and the system equations 
\begin{subequations}
\label{eq:numExp}
\begin{align}
	\dot  u - \tfrac{1}{10}\, \Delta u &= f \qquad\text{in } \Omega, \label{eq:numExp:a} \\
	\dot u - \nabla_\Gamma\cdot( a_\eps \nabla_\Gamma u) + \partial_\nu u &= g \qquad \text{on } \Gamma_\text{dyn}, \label{eq:numExp:b} \\
	u &= 0 \qquad \text{on } \partial\Omega\setminus\Gamma_\text{dyn}. \label{eq:numExp:c}
\end{align}
\end{subequations}
The values of $a_\eps$, the right-hand sides, the initial data, and the boundary parts are specified in the following examples. 
Since we focus on the spatial discretization error, we consider an implicit Euler discretization in time. Moreover, all approximations are computed on the same uniform temporal mesh as the respective reference solution, namely with time step size~$\tau=0.01$. 
Besides the LOD spaces for the approximation of~$p$, we consider uniform partitions of~$\Omega$ into quadrilaterals. This means that we approximate~$u$ by finite element functions of partial degree one, cf.~Remark~\ref{rem:Q1}.
%
%
\subsection{Smooth and discontinuous coefficients}\label{sect:numerics:exp1}
As in the theoretical part of the paper we consider dynamic boundary conditions on the entire boundary. In terms of~\eqref{eq:numExp} this means~$\Gamma_\text{dyn}=\Gamma=\partial\Omega$. 
As right-hand sides we consider~$f(t)\equiv 1$, $g(t)=t$ and the initial condition is defined by~$u_0(x,y) = \sin(\pi x) \cdot \cos(\frac 52 \pi y + 1)$. In the PDAE formulation, where we can choose~$p_0$ independently, we set~$p_0(x,y)=u_0(x,y)$ in a consistent manner.

For the diffusion coefficient~$a_\eps$ we compare the results for smooth but highly oscillatory and general discontinuous coefficients. For this, we define 
\[
  a_\eps^\text{sm}(x) 
  \coloneqq \frac{1}{ 2 + \cos(2\pi x \eps^{-1}) }
\] 
and~$a_\eps^\text{dc}$ by the piecewise constant (and thus discontinuous) function which takes random values in the range~$[\frac{1}{10}, 1]$ on a partition of mesh size~$\eps$. We emphasize that~$a_\eps^\text{dc} \not\in W^{1,\infty}(\Gamma)$, leading to a solution with~$p\in H^{1}(\Gamma)$ only. 
\begin{figure}
\includegraphics[width=.45\textwidth]{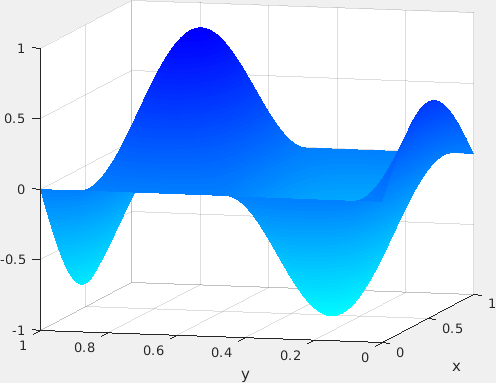}\hspace{0.5em}
\includegraphics[width=.45\textwidth]{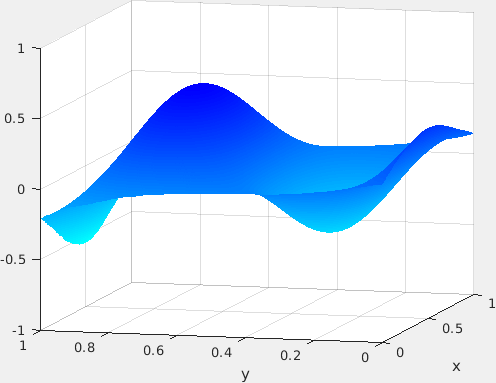}
\caption{Illustration of the solution~$u$ for~$t=0$ (left) and~$t=0.1$ (right) in the case of a random coefficient~$a_\eps^\text{dc}$ with~$\eps=2^{-9}$.}
\label{fig:exp1:solution}
\end{figure}

In this first example we apply the nodal interpolation operator such that the correctors are automatically localized and we have no localization parameter~$m$ to care about, cf.~Remark~\ref{rem:nodalInterpolation}. 
Apart from that, we consider the situation as described in Section~\ref{sect:disc:LOD}, where the meshes~$\calTO|_\Gamma$ and~$\calTG$ coincide and~$H_\Omega=H_\Gamma$ holds. 
The resulting numerical approximation for~$\eps=2^{-9}$ is shown in Figure~\ref{fig:exp1:solution}. The corresponding convergence plots for~$u-u_H$ and $p-\Pi_H \tilde p_H$ measured in the~$L^2$-norm for~$a_\eps^\text{sm}$ and~$a_\eps^\text{dc}$ are presented in Figure~\ref{fig:exp1:conv}. Note that we are in the range~$H \gtrsim \eps$. As a result, we observe poor results for the standard finite element approach as we are in the pre-asymptotic regime.  
On the other hand, the combination of Lagrange elements for~$u$ and a multiscale approach for~$p$ yields remarkable results: In the case of a smooth coefficient we even reach the full second-order rate. 
For the discontinuous coefficient, $u$ converges with second order whereas~$p$ shows an order of~$0.647$ (averaged over the last three mesh sizes). Note that this is slightly better than the shown bounds of Theorem~\ref{thm:main}, which equals~$0.5$ for~$\sigma=1$ and~$s=3/2$. 
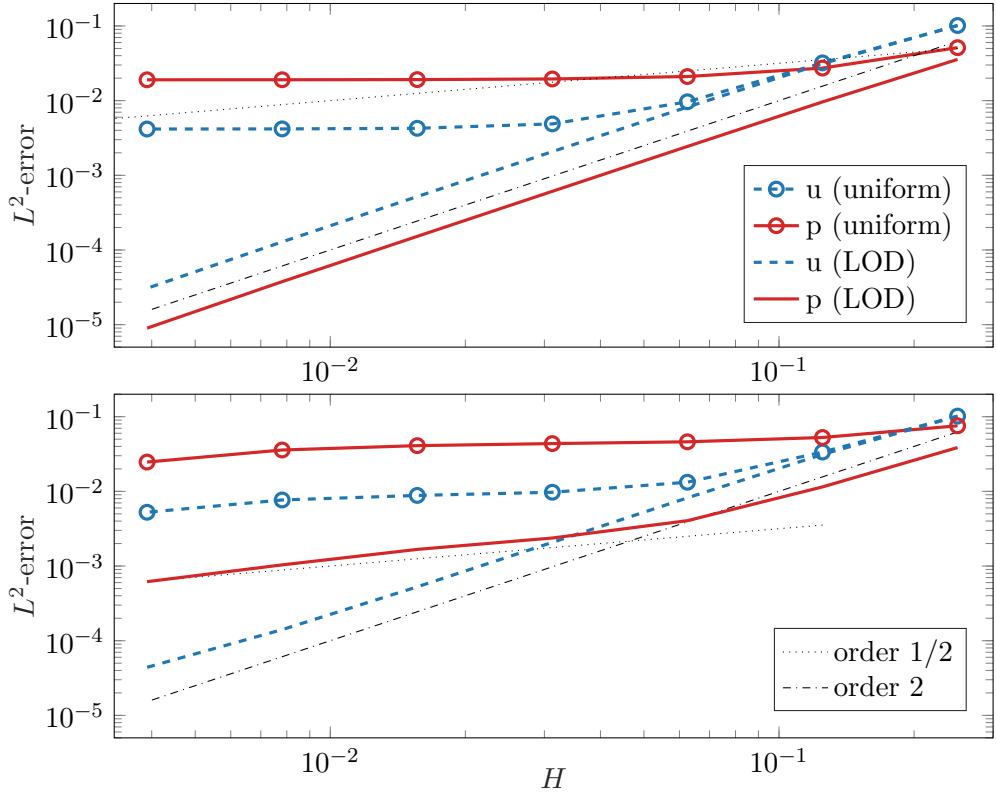
\begin{figure}
%
%
\begin{tikzpicture}

\begin{axis}[%
width=4.6in,
height=1.8in,
at={(0.758in,0.481in)},
scale only axis,
xmode=log,
xmin=0.0033,
xmax=0.3,
xminorticks=true,
xlabel style={font=\color{white!15!black}},
ymode=log,
ymin=5e-06,
ymax=0.2,
yminorticks=true,
ylabel style={font=\color{white!15!black}},
ylabel={$L^2$-error},
axis background/.style={fill=white},
legend style={at={(0.97, 0.53)},legend cell align=left, align=left, draw=white!15!black}
]

\addplot [color=color0, dashed, very thick, mark=o, mark options={solid, color0}, mark size=2.6]
  table[row sep=crcr]{%
0.25	0.101412805873472\\
0.125	0.032020938890403\\
0.0625	0.00965177851297862\\
0.03125	0.00488154900095019\\
0.015625	0.00426068383388044\\
0.0078125	0.00418162415470507\\
0.00390625	0.00416730798286368\\
};
\addlegendentry{u (uniform)}

\addplot [color=color3, very thick, mark=o, mark options={solid, color3}, mark size=2.6]
  table[row sep=crcr]{%
0.25	0.0512212418491566\\
0.125	0.0272883570002207\\
0.0625	0.0210348705774101\\
0.03125	0.0195087983894958\\
0.015625	0.0191316086748544\\
0.0078125	0.0190376243660371\\
0.00390625	0.0190141485325175\\
};
\addlegendentry{p (uniform)}

\addplot [color=color0, dashed, very thick]
table[row sep=crcr]{%
	0.25	0.101344430273117\\
	0.125	0.0313062590898206\\
	0.0625	0.00820968053932293\\
	0.03125	0.00207520740117321\\
	0.015625	0.000518899513558243\\
	0.0078125	0.000128399451522954\\
	0.00390625	3.06850232787514e-05\\
};
\addlegendentry{u (LOD)}

\addplot [color=color3, very thick]
  table[row sep=crcr]{%
0.25	0.035654180002709\\
0.125	0.00961283402603435\\
0.0625	0.0024352600956154\\
0.03125	0.000610287253827253\\
0.015625	0.000152256873936824\\
0.0078125	3.76380304677334e-05\\
0.00390625	8.97628007630684e-06\\
};
\addlegendentry{p (LOD)}

\addplot [color=black, dotted]
table[row sep=crcr]{%
	0.25	0.05\\
	0.125	0.0353553390593274\\
	0.0625	0.025\\
	0.03125	0.0176776695296637\\
	0.015625	0.0125\\
	0.0078125	0.00883883476483184\\
	0.00390625	0.00625\\
	0.001953125	0.00441941738241592\\
};



\addplot [color=black, dashdotted]
table[row sep=crcr]{%
	0.25	0.0625\\
	0.125	0.015625\\
	0.0625	0.00390625\\
	0.03125	0.0009765625\\
	0.015625	0.000244140625\\
	0.0078125	6.103515625e-05\\
	0.00390625	1.52587890625e-05\\
};

\addplot [color=gray]
  table[row sep=crcr]{%
0.001953125	1e-7\\
0.001953125	1\\
};


\end{axis}
\end{tikzpicture}%
%
%
\begin{tikzpicture}

\begin{axis}[%
width=4.6in,
height=1.8in,
at={(0.758in,0.481in)},
scale only axis,
xmode=log,
xmin=0.0033,
xmax=0.3,
xminorticks=true,
xlabel style={font=\color{white!15!black}},
xlabel={$H$},
xlabel style={below=-0.1in},
ymode=log,
ymin=5e-06,
ymax=0.2,
yminorticks=true,
ylabel style={font=\color{white!15!black}},
ylabel={$L^2$-error},
axis background/.style={fill=white},
legend style={at={(0.97, 0.32)},legend cell align=left, align=left, draw=white!15!black}
]

\addplot [color=black, dotted]
table[row sep=crcr]{%
	0.125	0.00353553390593274\\
	0.0625	0.0025\\
	0.03125	0.00176776695296637\\
	0.015625	0.00125\\
	0.0078125	0.000883883476483184\\
	0.00390625	0.000625\\
};
\addlegendentry{order $1/2$}


\addplot [color=black, dashdotted]
table[row sep=crcr]{%
	0.25	0.0625\\
	0.125	0.015625\\
	0.0625	0.00390625\\
	0.03125	0.0009765625\\
	0.015625	0.000244140625\\
	0.0078125	6.103515625e-05\\
	0.00390625	1.52587890625e-05\\
};
\addlegendentry{order $2$}

\addplot [color=color0, dashed, very thick, mark=o, mark options={solid, color0}, mark size=2.6]
  table[row sep=crcr]{%
0.25	0.101827344845498\\
0.125	0.0336211155752775\\
0.0625	0.0132516482426895\\
0.03125	0.00972981016690319\\
0.015625	0.00882069416739944\\
0.0078125	0.00766663405758851\\
0.00390625	0.00524972123116771\\
};

\addplot [color=color3, very thick, mark=o, mark options={solid, color3}, mark size=2.6]
table[row sep=crcr]{%
	0.25	0.0755575060788344\\
	0.125	0.0526375781719465\\
	0.0625	0.0460994780253115\\
	0.03125	0.0436387428733219\\
	0.015625	0.0409498603438385\\
	0.0078125	0.035811365147498\\
	0.00390625	0.0246808632497424\\
};

\addplot [color=color0, dashed, very thick]
  table[row sep=crcr]{%
0.25	0.10140660959236\\
0.125	0.0312284171631803\\
0.0625	0.0081823874038917\\
0.03125	0.00207434212046493\\
0.015625	0.000526797876821954\\
0.0078125	0.000140701434874882\\
0.00390625	4.40317139473479e-05\\
};

\addplot [color=color3, very thick]
  table[row sep=crcr]{%
0.25	0.0385986885745791\\
0.125	0.0114718985428273\\
0.0625	0.00404822831754504\\
0.03125	0.00237330527394362\\
0.015625	0.00167068140829583\\
0.0078125	0.00103639705864967\\
0.00390625	0.000617648102178747\\
};

\addplot [color=gray]
  table[row sep=crcr]{%
0.001953125	1e-7\\
0.001953125	1\\
};

\end{axis}
\end{tikzpicture}%
\caption{Convergence history in the $L^2$-norm for smooth (top) and discontinuous diffusion coefficients (bottom). Results shown for~$\eps=2^{-9}$.}
\label{fig:exp1:conv}
\end{figure}

Finally, we comment on the results measured in the $H^1$-norm. 
For the discontinuous coefficient~$a_\eps^\text{dc}$ we observe no convergence in~$p$ but convergence of order~$1/2$ (standard FEM) and almost~$1$ (LOD) in~$u$. 
This improves for the smooth coefficient, where also~$p$ converges with order~$1$ if multiscale finite elements are applied. 
%
%
\subsection{Mixed boundary conditions}\label{sect:numerics:exp2}
In the second experiment we mix two types of boundary conditions and consider dynamic boundary conditions only on~$\Gamma_\text{dyn} \coloneqq (0,1) \times \{0\}$. On the remaining parts we assume homogeneous Dirichlet boundary conditions. 
The input data is given by~$f(t)\equiv 1$, $g(t) \equiv 0$ with initial condition~$u_0(x,y) = \sin(\pi x) \cdot \cos(\frac 52 \pi y)$. Further, we consider a random coefficient~$a_\eps^\text{dc}$ with~$\eps=2^{-9}$ as described in the previous experiment.
We now consider the LOD as described in Section~\ref{sect:disc:LOD}, i.e., with a quasi-interpolation operator and correctors computed over patches $\textup{N}^m(T)$. 

The convergence results in the $L^2$-norm are very similar to the previous example such that we omit the details here. We mention, however, that convergence only takes place for sufficiently large~$m$, compared to the mesh size~$H$. Since we are interested in coarse mesh sizes, a localization parameter~$m=2$ usually yields satisfactorily results. 
To show the influence of the localization parameter~$m$ in more detail, we consider the error~$p-\tilde p_H$ in the~$H^1$-norm without the projection~$\Pi_H$. Recall from the previous example that~$p-\Pi_H\tilde p_H$ does not converge for the discontinuous diffusion coefficient. 
In Figure~\ref{fig:exp2:conv} one can observe that the ``full approximation'' $\tilde p_H$ also converges in the~$H^1$-norm. Further, one can see the limitation of the approximation for fixed~$m$, i.e., if the mesh size is no longer in the regime $m\gtrsim|\log H|$, the error stagnates or may even slightly grow with a further decrease of $H$. 
\begin{figure}
%
%
\begin{tikzpicture}

\begin{axis}[%
width=4.6in,
height=1.8in,
at={(0.758in,0.481in)},
scale only axis,
xmode=log,
xmin=0.007,
xmax=0.3,
xminorticks=true,
xlabel style={font=\color{white!15!black}},
xlabel={$H$},
xlabel style={below=-0.1in},
ymode=log,
ymin=3e-03,
ymax=2,
yminorticks=true,
ylabel style={font=\color{white!15!black}},
ylabel={$H^1$-error},
axis background/.style={fill=white},
legend style={at={(0.97, 0.37)},legend cell align=left, align=left, draw=white!15!black}
]

\addplot [color=color3, dotted, very thick]
  table[row sep=crcr]{%
0.25	1.15959282292278\\
0.125	1.11074685464759\\
0.0625	1.08634941740846\\
0.03125	1.07830444315791\\
0.015625	1.06183893689845\\
0.0078125	0.975299380602104\\
};
\addlegendentry{$p-\Pi_H \tilde p_H$} 

\addplot [color=color3, very thick]
  table[row sep=crcr]{%
0.25	0.198475224709606\\
0.125	0.0711171558786957\\
0.0625	0.0304909681165764\\
0.03125	0.0149245142715879\\
0.015625	0.00760872466220341\\
0.0078125	0.01029539302967\\
};
\addlegendentry{$p-\tilde p_H$} 



\addplot [color=black, dashdotted]
table[row sep=crcr]{%
	0.25	0.25\\
	0.125	0.125\\
	0.0625	0.0625\\
	0.03125	0.03125\\
	0.015625	0.015625\\
	0.0078125	0.0078125\\
};
\addlegendentry{order $1$}




\addplot [color=color3, very thick, mark=o, mark options={solid, color3}, mark size=2.6]
table[row sep=crcr]{%
	0.25	0.198475224709606\\
	0.125	0.144929700119228\\
	0.0625	0.155086328006263\\
	0.03125	0.218902731890255\\
	0.015625	0.256565829523912\\
	0.0078125	0.278717282053979\\	
};



\addplot [color=color3, very thick, mark=triangle, mark options={solid, color3}, mark size=2.6]
table[row sep=crcr]{%
	0.25	0.162508695667342\\
	0.125	0.0711171558786957\\
	0.0625	0.067600312592293\\
	0.03125	0.0937347281280486\\
	0.015625	0.107529009419803\\
	0.0078125	0.118562267838881\\
};



\addplot [color=color3, very thick, mark=square, mark options={solid, color3}, mark size=2.6]
table[row sep=crcr]{%
	0.25	0.152017499460782\\
	0.125	0.0453765569290689\\
	0.0625	0.0304909681165764\\
	0.03125	0.0395799260028487\\
	0.015625	0.0445645520705129\\
	0.0078125	0.0558723272387358\\	
};






\end{axis}
\end{tikzpicture}%
\caption[~]{Convergence history in the $H^1$-norm for a random coefficient with~$\eps=2^{-9}$ and different values of~$m$. The solid line without marks shows an approximation with increasing~$m$ (up to $m=5$). The remaining entries consider a fixed~$m=1$ (\tikz{\node[mark size=2.5pt,line width=1.1pt] at (0,0) {\pgfuseplotmark{o}};}), $m=2$ (\tikz{\node[mark size=3.3pt,line width=1.1pt] at (0,0) {\pgfuseplotmark{triangle}};}), or~$m=3$ (\tikz{ \node[mark size=2.8pt,line width=1.1pt] at (0,0) {\pgfuseplotmark{square}};}). }
\label{fig:exp2:conv}
\end{figure}
%
%
\subsection{Refinement of boundary}\label{sect:numerics:exp3}
In this final experiment we demonstrate the possibility of the PDAE approach to combine different meshes on~$\Omega$ and~$\Gamma$. 
We consider once more the mixed boundary case with~$\Gamma_\text{dyn} \coloneqq (0,1) \times \{0\}$. As diffusion coefficient we choose~$a_\eps^\text{sm}$ with a moderate~$\eps=1/4$. The initial data reads~$u_0(x,y) = \sin(3\pi x) \cdot \cos(\frac 52 \pi y + 1)$.
For this particular example we do not consider LOD spaces but standard~$\calP_1(\calTG)$ elements on the boundary. 
For the bulk we fix a uniform mesh~$\calTO$ with mesh size $H_\Omega$. 
Contrariwise, we apply uniform refinements on the boundary, i.e., we consider a mesh~$\calTG$ with mesh sizes~$H_\Gamma = H_\Omega, \frac 12 H_\Omega, \frac 14 H_\Omega, \dots$. 

The numerical results in Figure~\ref{fig:exp3:conv} indicate that this refinement has no positive effect on the approximation of~$u$. The boundary values~$p$, however, can be improved significantly. This does not only become evident for the $L^2$-norm but also in the $H^1$-norm. 
Recall that even the LOD approach considered in the previous two examples could only provide small $H^1$-errors if corrector functions were added. The here presented refinement of the boundary (without changing the interior mesh~$\calTO$) thus provides a tool to improve boundary approximations at low costs. 
\begin{figure}
%
%
\begin{tikzpicture}

\begin{axis}[%
width=2.35in,
height=2.1in,
at={(0.758in,0.53in)},
scale only axis,
xmode=log,
xmin=0.0007,
xmax=0.16,
xminorticks=true,
xlabel style={font=\color{white!15!black}},
xlabel={$H_\Gamma$},
ymode=log,
ymin=3e-05,
ymax=0.5,
yminorticks=true,
yticklabel pos=right,
ylabel style={font=\color{white!15!black}},
ylabel={$L^2$-error},
ylabel style={below=0.2in},
axis background/.style={fill=white},
]

\addplot [color=color0, mark=triangle, mark options={solid, color0}, mark size=2.6, very thick, dashed]
  table[row sep=crcr]{%
0.125	0.0346531094850169\\
0.0625	0.0337167265053244\\
0.03125	0.0335295966928827\\
0.015625	0.0334822866052367\\
0.0078125	0.0334704367975584\\
0.00390625	0.0334674733803953\\
0.001953125	0.0334667324727735\\
0.0009765625	0.0334665472426447\\
};


\addplot [color=color3, mark=triangle, mark size=2.6, very thick]
  table[row sep=crcr]{%
0.125	0.0194206784856389\\
0.0625	0.00737174586495099\\
0.03125	0.00422015008533206\\
0.015625	0.0033686214749949\\
0.0078125	0.00315330730004703\\
0.00390625	0.00309941238066237\\
0.001953125	0.00308593671255527\\
0.0009765625	0.00308256772195011\\
};



\addplot [color=color0, mark=*, mark options={solid, color0}, mark size=2.2, very thick, dashed]
table[row sep=crcr]{%
	0.015625	0.000627981191548936\\
	0.0078125	0.000609035974263876\\
	0.00390625	0.000605286303852835\\
	0.001953125	0.000604413510057278\\
	0.0009765625	0.000604199393932139\\
};


\addplot [color=color3, mark=*, mark size=2.2, very thick]
table[row sep=crcr]{%
	0.015625	0.000378300999635429\\
	0.0078125	0.000132452820389771\\
	0.00390625	7.12925269033217e-05\\
	0.001953125	5.62951019071951e-05\\
	0.0009765625	5.25949700917349e-05\\
};


\addplot [color=black, dashdotted]
table[row sep=crcr]{%
	0.125	0.0125\\
	0.0625	0.00625\\
	0.03125	0.003125\\
	0.015625	0.0015625\\
	0.0078125	0.00078125\\
	0.00390625	0.000390625\\
	0.001953125	0.0001953125\\
	0.0009765625	0.00009765625\\
};

\end{axis}
\end{tikzpicture}%
\hspace{-0.4cm}
%
%
\begin{tikzpicture}

\begin{axis}[%
width=2.35in,
height=2.1in,
at={(0.758in,0.53in)},
scale only axis,
xmode=log,
xmin=0.0007,
xmax=0.16,
xminorticks=true,
xlabel style={font=\color{white!15!black}},
xlabel={$H_\Gamma$},
ymode=log,
ymin=3e-05,
ymax=0.5,
yminorticks=true,
ylabel style={font=\color{white!15!black}},
ylabel={$H^1$-error},
ylabel style={below=2.9in},
yticklabels = {},
axis background/.style={fill=white},
legend style={at={(0.95,0.37)}, legend cell align=left, align=left, draw=white!15!black}
]


\addplot [color=color0, very thick, dashed]
table[row sep=crcr]{%
	0.125	0.174135094639363\\
};
\addlegendentry{$u-u_H$} 

\addplot [color=color3, very thick]
table[row sep=crcr]{%
	0.125	0.146558229742511\\
};
\addlegendentry{$p-p_H$} 

\addplot [color=black, dashdotted]
table[row sep=crcr]{%
	0.125	0.125\\
	0.0625	0.0625\\
	0.03125	0.03125\\
	0.015625	0.015625\\
	0.0078125	0.0078125\\
	0.00390625	0.00390625\\
	0.001953125	0.001953125\\
	0.0009765625	0.0009765625\\
};
\addlegendentry{order 1}


\addplot [color=color0, mark=triangle, mark options={solid, color0}, mark size=2.6, very thick, dashed]
  table[row sep=crcr]{%
0.125	0.174135094639363\\
0.0625	0.171698631201692\\
0.03125	0.171282285261744\\
0.015625	0.171183232888466\\
0.0078125	0.17115886046379\\
0.00390625	0.171152793931151\\
0.001953125	0.171151278993859\\
0.0009765625	0.171150900366076\\
};

\addplot [color=color3, mark=triangle, mark size=2.6, very thick]
  table[row sep=crcr]{%
0.125	0.231433675534673\\
0.0625	0.104660624220248\\
0.03125	0.0579022655083141\\
0.015625	0.0391820602733269\\
0.0078125	0.0329814770662258\\
0.00390625	0.0312484396240297\\
0.001953125	0.0308007734143397\\
0.0009765625	0.0306879414818688\\
};


\addplot [color=color0, mark=*, mark options={solid, color0}, mark size=2.2, very thick, dashed]
table[row sep=crcr]{%
	0.015625	0.0185213575180686\\
	0.0078125	0.0185139188977422\\
	0.00390625	0.0185127240698551\\
	0.001953125	0.0185124681676631\\
	0.0009765625	0.0185124068885239\\
};

\addplot [color=color3, mark=*, mark size=2.2, very thick]
table[row sep=crcr]{%
	0.015625	0.0214532939711467\\
	0.0078125	0.0106248436111328\\
	0.00390625	0.00519715321662344\\
	0.001953125	0.00237020398083026\\
	0.0009765625	0.000527059447663544\\
};

\end{axis}
\end{tikzpicture}%
\caption[~]{Convergence history in the $L^2$ (left) and $H^1$-norm (right) for $u$ and $p$, respectively. The different plots start with a mesh size~$H_\Omega=2^{-3}$ (\tikz{\node[mark size=3.5pt,line width=1.1pt] at (0,0) {\pgfuseplotmark{triangle}};}) and~$H_\Omega=2^{-6}$ (\tikz{ \node[mark size=2.5pt] at (0,0) {\pgfuseplotmark{*}};}). 	
The reference solution is computed on a uniform mesh with~$H_\Omega=H_\Gamma=2^{-10}$.}
\label{fig:exp3:conv}
\end{figure}
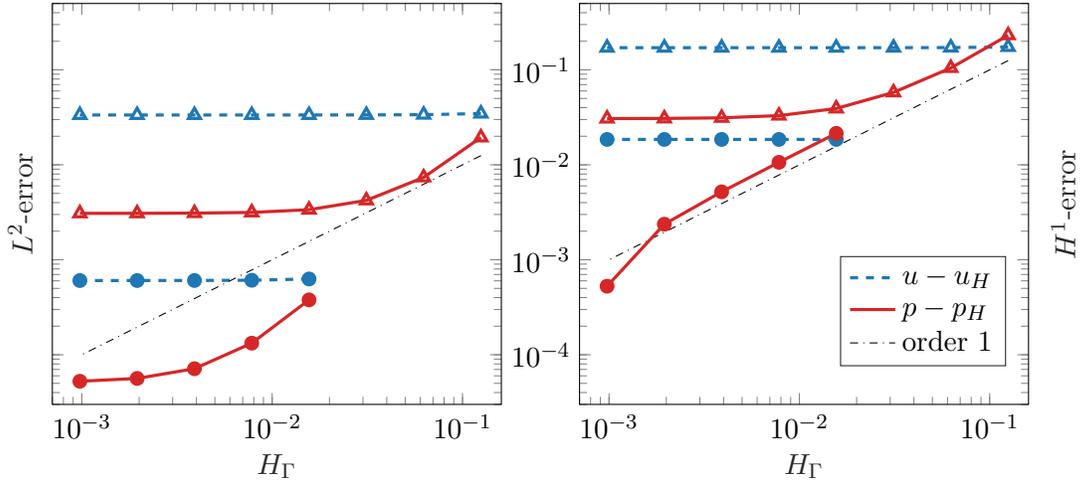
%
%
\section{Conclusion}\label{sect:conclusion}
In this paper, we have discussed the possibility of combining different approximation schemes in the bulk and on the boundary for parabolic problems with dynamic boundary conditions. 
In this way, we could consider multiscale techniques such as the LOD in combination with standard finite element spaces in the interior. 
We have shown analytically and observed numerically that this strategy allows remarkable speed-ups if low-regularity solutions are expected. This is the case for heterogeneous media as considered in this paper but may also be caused by nonlinearities. For this, the here presented schemes need to be combined with appropriate time discretization schemes, cf.~\cite{AltZ20}. 
The proposed decoupling approach may also be beneficial for the construction of splitting methods if bulk and surface dynamics have different time scales. 
%
%
\bibliographystyle{alpha} 
\bibliography{bib_augsburg}
%
%
\appendix
\section{Mosco convergence of the energy functional for \eqref{eq:coupledSys}}\label{app:Mosco}
In this appendix, we close the remaining gap in the proof of Theorem \ref{thm:limit}, namely the Mosco convergence of the energy functional associated with the elliptic part of \eqref{eq:coupledSys}.

We define the energy functional $\calE_\delta\colon H^1(\Omega)\times H^1(\Omega_\delta)\to \R$ via
\[\calE_\delta(u,w) =\int_{\Omega}\frac12 |\nabla u|^2\dx + \int_{\Omega_\delta}\frac{1}{2\delta} a_\eps(x)|\nabla w|^2\dx.
\]
We now transform the variable domain $\Omega_\delta$ to a fixed domain.
Due to the assumed smoothness of $\Gamma$ and for sufficiently small $\delta$, every point $x\in \Omega_\delta$ can uniquely be written in the form $x=X_\delta(y,\theta)\coloneqq y+\delta \theta \nu(y)$ for $y\in \Gamma$ and $\theta\in (0,1)$, where $\nu$ denotes the outer unit normal of $\Omega$.

With this change of coordinates we define $\Sigma\coloneqq \Gamma\times(0,1)$ and for a function $w\colon \Omega_\delta \to \R$ we set $W=w\circ X_\delta\colon \Sigma\to\R$. Because of the smoothness of $\Gamma$ we have $W\in H^1(\Sigma)$ whenever $w\in H^1(\Omega_\delta)$ and the gradients can be computed by the chain rule.
By slight abuse of notation, we denote $a_\eps\circ X_\delta$ again by $a_\eps$ and note that it only depends on $y$, not on $\theta$.
Moreover, we introduce the space $\calW\coloneqq\{(u,W)\in H^1(\Omega)\times H^1(\Sigma)\,|\,u|_\Gamma=W|_{\{\theta=0\}}\}$.
Inserting these transformations into $\calE_\delta$, we arrive at the energy functional $E_\delta\colon \calW\to R$ defined by 
\begin{equation}\label{eq:defenergy}
  E_\delta (u,W)
  = \int_{\Omega}\frac12\, |\nabla u|^2\dx 
  + \int_{\Sigma}\frac{1}{2\delta}\, a_\eps(y) \Big(\nabla_\Gamma  W\cdot \mathbf{B}_\delta (y,\theta)\nabla_\Gamma W+\frac{1}{\delta^2}|\partial_\theta W|^2 \Big) \mathbf{J}_\delta(y,\theta)\, \text{d}(y,\theta).
\end{equation}
Here, $\mathbf{B}_\delta$ describes the transformation of the (tangential part of) the gradient and $\mathbf{J}_\delta$ the volume change.
Lemma 2.2 in \cite{Lie13} states that $\mathbf{B}_\delta \to \operatorname{Id}$ and $\mathbf{J}_\delta/\delta\to 1$, both uniformly in $\delta$.

The main goal of this appendix is to show the following Proposition concerning the Mosco convergence of $E_\delta$, which is the analogue to \cite[Thm.~3.2]{Lie13}.
\begin{proposition}
The energy functional $E_\delta$ converges in the sense of Mosco to the limit functional $E_0\colon\calW\to \R$ given by
\begin{equation*}
E_0(u,W)\coloneqq
\begin{cases}
\int_{\Omega}\frac12 |\nabla u|^2\dx + \int_{\Sigma}\frac{1}{2} a_\eps(y)|\nabla_\Gamma  W|^2\, \text{d}(y,\theta) \qquad \text{if}\quad (u,W)\in \calW_0,\\
+\infty \qquad \text{else},
\end{cases}
\end{equation*}
where $\calW_0\coloneqq\{(u,W)\in \calW\, |\, \partial_\theta W=0\}$.
\end{proposition}
Mosco convergence is Gamma convergence in the strong and weak topology simultaneously, see \cite[Sect.~3]{Lie13}
Hence, we have to show (i) a liminf-estimate for a weakly converging sequence and (ii) a limsup-estimate for a strongly converging recovery sequence.

\begin{proof}
\emph{(i):} Let $(u_\delta, W_\delta)\wconv (u,W)$ in $\calW$. Assuming that $\lim \inf_{\delta\to 0}E_\delta(u_\delta, W_\delta)<\infty$, we necessarily have $(u, W)\in\calW_0$ due to the weak lower semicontinuity of the norm on $\calW$.
It holds that
\[E_\delta(u_\delta, W_\delta)\geq  \int_{\Omega}\frac12 |\nabla u_\delta|^2\dx + \int_{\Sigma}\frac{1}{2} a_\eps(y)\nabla_\Gamma  W_\delta\cdot \mathbf{B}_\delta (y,\theta)\nabla_\Gamma W_\delta \,\frac{1}{\delta}\mathbf{J}_\delta(y,\theta)\, \text{d}(y,\theta).\]
The liminf-estimate now follows from the uniform convergence of $\mathbf{B}_\delta$ and $\mathbf{J}_\delta/\delta$.

\emph{(ii):} For $(u,W)$ we choose the constant recovery sequence $(u_\delta, W_\delta)=(u,W)$. 
In the case $(u,W)\not \in \calW_0$ the result is trivial, since $E_0(u, W)=\infty$ and we can argue as in (i).
In the case $(u,W)\in \calW_0$, the derivative $\partial_\theta$ in $E_\delta$ vanishes and we obtain
\[E_\delta(u,W)=\int_{\Omega}\frac12 |\nabla u|^2\dx + \int_{\Sigma}\frac{1}{2} a_\eps(y)\nabla_\Gamma  W\cdot \mathbf{B}_\delta (y,\theta)\nabla_\Gamma W\,\frac{1}{\delta}\mathbf{J}_\delta(y,\theta)\, \text{d}(y,\theta)\to E_0(u,W),\]
arguing as above.
\end{proof}

Note that $(u,W)\in \calW_0$ if and only if $(u,W)\in H^1(\Omega)\times H^1(\Gamma)$ with $u|_\Gamma=W$. Hence, the limit energy $E_0$ can be reduced by integrating over $\theta$, which exactly gives the energy functional associated with the elliptic part of \eqref{eq:dynBC}.
\end{document}